\documentclass[12pt]{amsart}

\usepackage{amssymb}
\usepackage{amsfonts}
\usepackage{amsmath}

\usepackage{graphicx}
\usepackage{xcolor}
\usepackage{mathrsfs}
\usepackage{stmaryrd}
\usepackage{epsfig,color}
\usepackage{blindtext}
\usepackage{enumerate}
\usepackage{hyperref}
\usepackage{url}
\usepackage{bbm}
\usepackage{filecontents}
\usepackage{nicefrac,mathtools}
\usepackage{bm}   
\usepackage[nocompress]{cite}
\usepackage{color}
\DeclareGraphicsExtensions{.pdf,.jpeg,.png}
\usepackage{epstopdf}
\usepackage{cancel} 
\usepackage[normalem]{ulem} 
\usepackage{verbatim} 
\usepackage{enumitem} 
\usepackage{tikz-cd}
\usetikzlibrary{cd}

\usepackage[msc-links, lite]{amsrefs}

\usepackage{geometry}
\usepackage{tikz}
\usetikzlibrary{decorations.markings}
\usetikzlibrary{arrows.meta}

\usepackage{extarrows}
\newtheorem{theorem}{Theorem}[section]
\newtheorem{conjecture}[theorem]{Conjecture}
\newtheorem{proposition}[theorem]{Proposition}
\newtheorem{lemma}[theorem]{Lemma}
\newtheorem{corollary}[theorem]{Corollary}

\newtheorem{example}[theorem]{Example}

\newtheorem{claim}[]{Claim}

\theoremstyle{definition}
\newtheorem{definition}[theorem]{Definition}

\newtheorem{remark}[theorem]{Remark}
\numberwithin{equation}{section}
\newcommand{\dv}{\mathrm{div}}

\newcommand{\mb}{\mathbb}
\newcommand{\mc}{\mathcal}

\newcommand{\oli}{\overline}

\newcommand{\wti}{\widetilde}

\newcommand{\Sc}{\mathrm{Sc}}
\newcommand{\mr}{\mathrm}

\newcommand{\n}{\mathbf n}

\DeclareMathOperator{\Ric}{Ric}

\DeclareMathOperator{\Div}{div}
\numberwithin{equation}{section}

\newcommand{\R}{\mathbb{R}}
\newcommand{\tr}{\mathrm{tr}}

\newcommand{\sph}{\mathbb{S}}
\newcommand{\normal}{\mathbf{n}}

\DeclareMathOperator{\radsphere}{Rad_{\mathbb{S}^n}}
\DeclareMathOperator{\radspherek}{Rad_{\mathbb{S}^{n-1}}}

\newcommand{\interior}[1]{%
	{\kern0pt#1}^{\mathrm{\,o}}%
}

\makeatletter
\let\save@mathaccent\mathaccent
\newcommand*\if@single[3]{%
	\setbox0\hbox{${\mathaccent"0362{#1}}^H$}%
	\setbox2\hbox{${\mathaccent"0362{\kern0pt#1}}^H$}%
	\ifdim\ht0=\ht2 #3\else #2\fi
}

\makeatother

\makeatletter
\newcommand*{\transpose}{%
	{\mathpalette\@transpose{}}%
}
\newcommand*{\@transpose}[2]{%
	\raisebox{\depth}{$\m@th#1\intercal$}%
}
\makeatother
	\usetikzlibrary{hobby}

 \usetikzlibrary{decorations}
\makeatletter
\def\pgfutil@Repeat#1#2{#2\ifnum#1>0
\expandafter\pgfutil@firstofone\else\expandafter\pgfutil@gobble\fi
{\expandafter\pgfutil@Repeat\expandafter{\the\numexpr#1-1\relax}{#2}}}
\tikzset{
  dash between/.code args={#1 and #2}{%
    \tikz@addoption{%
      \pgfgetpath\currentpath
      \pgfprocessround{\currentpath}{\currentpath}%
      \pgf@decorate@parsesoftpath{\currentpath}{\currentpath}%
      \pgfmathsetlengthmacro\firstpart{(#1)*\pgf@decorate@totalpathlength}%
      \pgfmathsetlengthmacro\secondpart{(#2-(#1))*\pgf@decorate@totalpathlength}%
      \pgfmathsetlengthmacro\thirdpart{(1-(#2))*\pgf@decorate@totalpathlength}%
      \edef\thirdpart{{\thirdpart}{0pt}}%
      \edef\firstpart{{\firstpart}{0pt}}%
      \pgfmathsetlengthmacro\secondpartlength{\pgfkeysvalueof{/tikz/dash between on}
                                            +(\pgfkeysvalueof{/tikz/dash between off})}%
      \pgfmathtruncatemacro\repetitions{\secondpart/\secondpartlength}%
      \pgfmathsetlengthmacro\secondexpand{\secondpart/\repetitions-\secondpartlength}%
      \edef\secondexpand{\the\dimexpr\pgfkeysvalueof{/tikz/dash between off}+\secondexpand\relax}%
      \edef\secondpart{%
        \pgfutil@Repeat{\the\numexpr\repetitions-1\relax}%
          {{\pgfkeysvalueof{/tikz/dash between on}}{\secondexpand}}%
      }%
      \edef\tikz@temp{\firstpart\secondpart\thirdpart}%
      \expandafter\pgfsetdash\expandafter{\tikz@temp}{+0pt}%
    }
  }
}
\makeatother
\tikzset{
  dash between style/.is choice,
  dash between style/dotted/.style        ={dash between on=\pgflinewidth,dash between off=2pt},
  dash between style/densely dotted/.style={dash between on=\pgflinewidth,dash between off=1pt},
  dash between style/loosely dotted/.style={dash between on=\pgflinewidth,dash between off=4pt},
  dash between style/dashed/.style        ={dash between on=3pt,dash between off=2pt},
  dash between style/loosely dashed/.style={dash between on=3pt,dash between off=6pt},
  dash between style/densely dashed/.style={dash between on=3pt,dash between off=2pt},
  dash between style/no/.style={dash between on=0pt, dash between off=1pt},
  dash between on/.initial=\pgflinewidth,
  dash between off/.initial=2pt,
  middle dotted line/.style={
    thick,
    dash between=.35 and .65}}
\begin{document}
	
\title[Scalar-mean rigidity theorem]{Scalar-mean  rigidity theorem for compact manifolds with boundary}
\author{Jinmin Wang}
\address[Jinmin Wang]{Institute of Mathematics, Chinese Academy of Sciences}
\email{\url{jinmin@amss.ac.cn}}
\thanks{}

\author{Zhichao Wang}
\address[Zhichao Wang]{Shanghai Center for Mathematical Science, 2005 Songhu Road, Fudan University, Shanghai, 200438, China}
\email{\url{zhichao@fudan.edu.cn}}
\thanks{}

\author{Bo Zhu}
\address[Bo Zhu]{Department of Mathematics, Texas A\&M University}
\email{\url{bozhu@tamu.edu}}
\thanks{The third author is partially supported by NSF 1952693, 2247322 and AMS-Simons Travel Grant}

\begin{abstract}
We prove a scalar-mean rigidity theorem for compact Riemannian manifolds with boundary in dimension less than five by developing a dimension reduction argument for mean curvature, which extends Schoen-Yau's dimension reduction argument for scalar curvature. As a corollary, we prove the sharp spherical radius rigidity theorem and best NNSC fill-in in terms of the mean curvature. Moreover, we prove a Lipschitz Listing type scalar-mean rigidity theorem for these dimensions. 
 \end{abstract}
\maketitle

\section{Introduction}
Comparison geometry is a significant topic in metric geometry and geometric analysis. The study of Ricci curvature and sectional curvature in comparison geometry has made substantial progress (see  \cites{Li_geometric_analysis, Schoen_Yau_lectures,Yau_perspective_geometric_analysis, Comparison_geometry,Gromov_four_lectures}). However, the corresponding problems related to the scalar curvature remain understudied. Recently, Gromov proposed to study topics related to the scalar curvature and its companion, mean curvature, in \cite{Gromov_mean_light_scalar}. Currently, using the (higher) index theory on \textbf{spin} Riemannian manifolds (see \cites{Stolz_psc,Gromov-Lawson_Dirac,Yu_zero_psc,Rosen_psc_novikov}) and $\mu$-bubble (soap bubble) (see \cites{Schoen_Yau_psc_higher,Schoen_Yau_imcompressible,Schoen_Yau_lectures,Chodosh_Li_soap_bubble,Zhu_Width,Liok_Zhu_cycle,Gromov_5d}) in Riemannian manifolds are two important tools for studying the geometry and topology of Riemannian manifolds with scalar curvature constraints.

\vspace{2mm}
Let us start with the following scalar curvature rigidity theorem on smooth, closed, \textbf{spin} Riemannian manifolds.
\begin{theorem} \label{thm: Llarull_Listing}
    Suppose that $(M^n,g)$ is a closed, smooth, \textbf{spin} Riemannian manifold and $F\colon (M,g)\to(\sph^n,g_{\sph^n})$ is a smooth map of $\deg(F) \neq 0$\footnote{$f_*([M]) = \deg(F) [\mathbb{S}^n]$}. 
    \begin{enumerate}
        \item $( $\text{Llarull}, see   \cite{Llarull}*{Theorem B}$)$ If
    $\|\wedge^2\mr dF\|\leq 1,  \Sc_g\geq n(n-1),$
     then $F$ is an isometry. Here, $\|\wedge^2\mr d F\|$ is the norm of $\wedge^2\mr dF\colon \wedge^2TM\to \wedge^2T\sph^n$,
\vspace{1mm}
    \item \label{item: Listing} $($\text{Listing}, see  \cite{Listing:2010te}*{Theorem 2}$)$ If  $\Sc_g\geq \|\wedge^2\mr dF\|\cdot n(n-1)$, then $F$ is an isometry.
    \end{enumerate}
    
\end{theorem}

Gromov proposes to study the geometry and topology of the mean curvature alongside scalar curvature (see \cites{Gromov_mean_light_scalar}). The scalar curvature rigidity theorem has been generalized to the scalar-mean rigidity theorem for compact, \textbf{spin} Riemannian manifolds with nonempty boundary  using index theory techniques. For example, a recent series of works \cites{Lottboundary,Wang:2022vf,MR3257837,Simone24,Wang:2021tq} have proved the scalar-mean rigidity theorem for smooth, compact, \textbf{spin} Riemannian manifolds with non-empty boundary.  Suppose that $(M^n, \partial M, g)$ is a smooth, compact, \textbf{spin} Riemannian manifold with nonnegative scalar curvature $\Sc_g \geq 0$ and uniformly positive mean curvature  $H_{\partial M} \geq n-1$\footnote{ $H_{\partial M}$ means the mean curvature of $\partial M$. For example, the mean curvature of unit $(n-1)$-sphere in the unit $n$-ball is equal to $(n-1)$.}. If $F: (\partial M, g_{\partial M}) \rightarrow (\mathbb{S}^{n-1},g_{\mathbb{S}^{n-1}})\footnote{$(\mathbb{S}^{n-1},g_{\mathbb{S}^{n-1}})$ is the standard unit $(n-1)$-sphere in $\mathbb{R}^{n}$.}$ is a distance non-increasing  map of $\deg(F)  \neq 0$, then $F$ is an isometry. Indeed, such scalar-mean rigidity holds similarly for more general manifolds with non-negative curvature operator and non-negative second fundamental form (see \cite{Lottboundary}*{Theorem 1.1}). Moreover, in the spin setting, the scalar-mean comparison results also hold for special domains in the warped product metric (see \cites{Cecchini:2021vs,ChaiWan24} for details).

\vspace{2mm}

Moreover, Gromov conjectures that the scalar-mean rigidity theorem holds without the spin assumption and suggests the approach of the capillary $\mu$-bubble (see \cite{Gromov_four_lectures}*{Section 5.8.1} for details). In this paper, without relying on any index theory techniques such as those in \cites{Lottboundary,Wang:2022vf,MR3257837,WangXieEu,Simone24}, we make use of the capillary $\mu$-bubble techniques together with dimension reduction for mean convex boundary and then prove a scalar-mean rigidity theorem for smooth compact Riemannian manifolds with smooth map $F$ as follows. 

\begin{theorem} \label{mainthm: scalar_mean_curvature_comparsion}
    Suppose that $(M^n, \partial M, g)$, $n=2, 3, 4$ is a smooth compact Riemannian manifold with non-negative scalar curvature $\Sc_g \geq 0$ and uniformly positive mean curvature $H_{\partial M} \geq n-1$. If $F: (\partial M, g_{\partial M}) \rightarrow (\mathbb{S}^{n-1}, g_{\mathbb{S}^{n-1}})$ is a distance non-increasing smooth map of $\deg(F)  \neq 0$, then 
    \begin{enumerate}
        \item $F$ is an isometry,
        \item $(M, g)$ is isometric to $(\mathbb{D}^n, g_{\mathbb{D}^n})\footnote{$(\mathbb{D}^n, g_{\mathbb{D}^n})$ means the standard unit disk in $\mathbb{R}^n$.}$.
    \end{enumerate}  
\end{theorem}

\vspace{1mm}
Recall that the capillary $\mu$-bubble is utilized by Li to prove the dihedral rigidity theorem for compact Riemannian manifolds with nonnegative scalar curvature, nonnegative mean curvature, and (certain) dihedral angle conditions (see \cites{Li_polyhedron_three,Li_n_prisms}); Chai-Wang also uses the capillary $\mu$-bubbles to prove scalar-mean rigidity of certain three-dimensional warped product spaces (see \cite{ChaiWang23}). However, our primary contribution is to develop the technique to study how the positive mean curvature, coupled with a nonzero degree map, inherits sharply under the process of \emph{dimension reduction} and then generalizes the scalar-mean rigidity theorem to higher dimensions without the spin assumption. Our main argument is essentially inspired by Schoen-Yau dimension reduction  for scalar curvature (see \cites{Schoen_Yau_psc_higher, Li_n_prisms,Gromov_Zhu_area}), and it can be viewed as a dimension reduction for mean curvature. 

\vspace{2mm}

As a further application, the scalar-mean curvature rigidity theorem \ref{mainthm: scalar_mean_curvature_comparsion} derives the following extremality results of the spherical radius and the best NNSC filling.

\begin{enumerate}
    \item Recall that the spherical radius of a Riemannian manifold $(N^n, g)$ is defined as
\begin{equation*}
    \radsphere(N,g) = \sup\{r: F: (N, g) \rightarrow (\mathbb{S}^n(r),g_{\sph^{n}(r)}), \ \|\mr dF\| \leq 1 \text{ and } \deg(F) \neq 0\}.
\end{equation*}
\begin{corollary} \label{corollary: spherical_radii}
   If $(M^n, \partial M, g), n=2, 3, 4$ is a smooth closed Riemannian manifold with non-negative scalar curvature $\Sc_g \geq 0$ and uniformly positive mean curvature $H_{\partial M} \geq n-1$, then $$\radsphere(\partial M, g_{\partial M}) \leq 1.$$ Moreover, the equality holds if and only if $(M^n, g)$ is isometric to $(\mathbb{D}^n, g_{\mathbb{D}^n})$.
\end{corollary}

\item Recall that Shi-Wang-Wang-Zhu \cite{Fill-in-SWWZ} prove that: If $(M^n, \partial M, g), 2 \leq n \leq 7$ is a smooth, compact Riemannian manifold with nonnegative scalar curvature  $\Sc_g \geq 0$ in $M$, then there exists a constant $c$ depending only on the intrinsic geometry of the boundary $\partial M$ such that
$$H_{\partial M} \leq c.$$ 
Here, we obtain a sharp constant as follows.
\begin{corollary} \label{corollary: Best_fill_in}
    Suppose that $(N^{n-1}, h), n =2, 3, 4$ is a closed smooth Riemannian manifold of dimension $n=2, 3, 4$. If $(M^n, \partial M, g)$ is a compact, non-negative scalar curvature fill-ins of $(N,h)$, then
    \begin{equation*}
        \inf_{p \in N} H_{\partial M }(p) \leq \frac{n-1}{\radspherek(N)}.
    \end{equation*}
Moreover, the equality holds if and only if $(M^n, g)$ is isometric to $(\mathbb{D}^n, g_{\mathbb{D}^n})$.

Recall that $(M^n, \partial M, g)$ is said to be a nonnegative scalar curvature fill-ins of $(N^{n-1},h)$ if $(M^n, \partial M, g)$ is a compact  manifold such that
\begin{equation*}
     \partial M = N,\ \ \Sc_g \geq 0, \ \
     g_{|N} = h.
\end{equation*}

\end{corollary}
\end{enumerate}

Moreover, the capillary $\mu$-bubble technique combined with the dimension reduction argument for mean curvature can also be applied to solve the Listing-type scalar-mean rigidity theorem for $n=2, 3, 4$, which is stronger than Theorem \ref{mainthm: scalar_mean_curvature_comparsion} in the sense of a more flexible mean curvature assumption on the boundary.

\begin{theorem}[Listing-type scalar-mean rigidity theorem] \label{thm: generalized_scalar_mean_comparsion}
  Suppose that $(M^n, \partial M, g)$, $ n =2, 3, 4$ is a smooth compact Riemannian manifold with non-negative scalar curvature $Sc_g \geq 0$ and mean convex boundary $H_{\partial M} > 0$. Let $F: (\partial M, g_{\partial M}) \rightarrow (\mathbb{S}^{n-1}, g_{\sph^{n-1}})$ be a smooth map with $\deg(F) \neq 0$. If $H_{\partial M} \geq \|\mr dF\|(n-1)$, then $F$ is a homothety and $(M^n,  cg)$ is isometric to $(\mathbb{D}^n, g_{\mathbb{D}^n})$ for some $c>0$. 
\end{theorem}

\begin{remark}
    In contrast, the Listing type theorem for a closed Riemannian manifold with scalar curvature lower bound remains open without a spin assumption (see (\ref{item: Listing}) in Theorem \ref{thm: Llarull_Listing} in this paper and \cite{CWXZ_Llarull_4} for the details). This highlights the differences between scalar curvature geometry and mean curvature geometry.
\end{remark}

\vspace{2mm}
Finally, to answer the rigidity theorems of Corollary \ref{corollary: spherical_radii} and Corollary \ref{corollary: Best_fill_in}, we prove the following Lipschitz scalar-mean rigidity theorem, which is a parallel development in geometric analysis in comparison with that in the spin settings (see \cites{Simone24, cecchini2022lipschitz}).

\begin{theorem}\label{thm:Lipschitz}
      Suppose that $(M^n, \partial M, g), n=2, 3, 4$ is a smooth compact Riemannian manifold with non-negative scalar curvature $\Sc_g \geq 0$ and uniformly positive mean curvature $H_{\partial M} \geq n-1$. If $F: (\partial M, g_{\partial M}) \to (\mathbb{S}^{n-1}, g_{\mathbb{S}^{n-1}})$ is a distance non-increasing \textbf{Lipschitz} map of $\deg(F)  \neq 0$, then 
      \begin{enumerate}
          \item    $F$ is a smooth isometry, 
          \item $(M, g)$ is isometric to $(\mathbb{D}^n, g_{\mathbb{D}^n})$.
      \end{enumerate}
   
\end{theorem}

Note that Theorem \ref{thm:Lipschitz} is stronger than Theorem \ref{mainthm: scalar_mean_curvature_comparsion}. We present them separately to aid the reader's understanding of the underlying ideas. Theorem \ref{mainthm: scalar_mean_curvature_comparsion} provides greater geometric insight, whereas Theorem \ref{thm:Lipschitz} is more technical in nature. The latter is primarily motivated by the goal of addressing the rigidity aspect in Corollary \ref{corollary: spherical_radii} and Corollary \ref{corollary: Best_fill_in} in this paper.

\subsection*{Remark on \emph{dimension reduction}} 

It should be noted that the scalar-mean rigidity theorems remain open for smooth, compact, \textbf{nonspin} Riemannian manifolds of dimension greater than four, due to the insufficient regularity of the capillary hypersurfaces near the boundary in higher dimensions. However, our \emph{dimension reduction} argument for the scalar-mean curvature rigidity theorem can still be applied, provided that the regularity of the capillary hypersurfaces has been improved in a generic sense. In contrast, the corresponding (Schoen-Yau) dimension reduction for scalar curvature, as used in \cite{CWXZ_Llarull_4}, cannot be effectively utilized in the proof of Llarull's theorem, where the regularity issue of the $\mu$-bubble for higher dimensions has already been resolved. This is the primary reason why Llarull's theorem has been confirmed only for $n=4$ in \cite{CWXZ_Llarull_4}. Consequently, this distinction underscores the differences between the scalar curvature problem and the mean curvature problem.

\subsection*{Proof outlines} Our primary technique involves the use of capillary $\mu$ bubble and dimension reduction, incorporating mean curvature and scalar curvature properties. Notably, the capillary $\mu$-bubble functional $\mathcal{A}_c$ (see Section \ref{sec: mu-bubble} or Appendix \ref{sec: capillarysurface}) has no nontrivial minimizer in the rigidity model $(\mathbb{D}^n, g_{\mathbb{D}^n})$. This presents a dilemma: perturbing the metric $g$ on $(M,g)$ to ensure the existence of a capillary $\mu$-bubble causes us to lose information about the scalar and mean curvature, thus yielding only the scalar-mean extremal theorem instead of the scalar-mean rigidity theorem. However, finding a smooth capillary $\mu$-bubble is crucial to initiating the dimension reduction argument in our context.

\vspace{1mm}
To overcome this difficulty, we introduce the trace norm $|dF|_{tr}$ of the map $F: (\partial M, g_{\partial M}) \to (\mathbb{S}^{n-1}, g_{\sph^{n-1}})$ in Section \ref{sec: mean_curvature-degree}, and then establish the connection between mean curvature and the degree of the map $F$. Roughly speaking, we demonstrate that a large mean curvature on the boundary $\partial M$, as expressed in terms of the trace norm of the map $F$, forces the degree of $F$ to vanish, thereby leading to the scalar-mean extremality theorem (see Proposition \ref{prop:weak}). The scalar-mean extremality lemma has two key aspects:

\begin{enumerate}
    \item It ensures that we can perturb the map $F$ to a new map (still denoted by $F$) that maps two small open neighborhoods in $\partial M$ to the poles $\{\pm p\}$ of $\mathbb{S}^{n-1}$, respectively. This process does not disrupt the Riemannian structure of $(M^n, \partial M, g)$ and guarantees that the minimizing problem for the $\mu$-bubble functional has no barriers (see Lemma \ref{lemma: existence_mu_bubble}). As a result, recent advances in the regularity of the capillary $\mu$-bubble apply in our context (see \cite{chodosh2024_improvedregularity}*{Theorem 1.1}).

\item  Using the scalar-mean extremality lemma, the dimension reduction technique, and the conformal metric technique that exchanges the scalar curvature with the mean curvature, we first prove the following in Section \ref{sec: proof_main_thm}:
\begin{itemize}
    \item $\textbf{Claim A: }\ \Sc_g = 0\text{ on }M;\  H = n-1=\|\mr dF\|_\tr\text{ on }\partial M,$\\
    under the assumption of Theorem \ref{mainthm: scalar_mean_curvature_comparsion}.

\item  $\textbf{Claim B: }\ \Sc_g = 0\text{ on }M;\ H_{\partial M } = \|\mr dF\|_{\tr} = \|\mr dF\|(n-1)\text{ on }\partial M,$\\
under the assumption of Theorem \ref{thm: generalized_scalar_mean_comparsion}.
\end{itemize}
Then we note that \textbf{Claim A} implies that $F$ is an isometry. Hence, Theorem \ref{mainthm: scalar_mean_curvature_comparsion} follows from the Shi-Tam inequality (see \cite{Shi_Tam}*{Theorem 1} for $n=3$ and \cite{Shi_Tam_extension}*{Theorem 2} for $n =4$ or refer to Appendix \ref{sec: Shi-Tam} for the precise statements in this paper). However, \textbf{Claim B} does not directly lead  to the application of  the Shi-Tam inequality. To address the difficulty, we make a conformal change to $(M^n, \partial M, g)$ with a suitable harmonic function with appropriate Neumann boundary condition in $\partial M$. Following this, we apply the Shi-Tam inequality to conclude that $F$ is an isometry.

\end{enumerate}
\vspace{1mm}

Finally, the extremity parts in Corollaries \ref{corollary: spherical_radii} and \ref{corollary: Best_fill_in} follow directly from Theorem \ref{mainthm: scalar_mean_curvature_comparsion}. However, the map  $F: (\partial M, g) \to (\sph^{n-1}, g_{\sph^{n-1}})$ attaining extremality is only a Lipschitz map, which leads to a lack of regularity in general. As a result, Theorem \ref{mainthm: scalar_mean_curvature_comparsion} cannot be applied directly. To address this issue, we introduce a stronger trace function $[dF]_\tr$ in oriented vector spaces, instead of the trace norm.  Using this, we prove the following in Section \ref{sec: Lip_scalar_mean_rigdity},
\begin{itemize}
    \item $\textbf{Claim C:}\ \ \ \ \  \Sc_g=0\text{ on }M,\  H_{g}=[\mr dF]_\tr=\|\mr dF\|_\tr=n-1\text{ on }\partial M$.
\end{itemize}

 Under the assumption of Theorem \ref{thm:Lipschitz}.
 We further prove that $F$ is an orientation-preserving map almost everywhere (see Lemma \ref{lemma:orientedTrace}). We  then conclude that $F$ is a smooth (Riemannian) isometry by using the results in \cite{cecchini2022lipschitz} and \cite{Myers_Steenrod}

\subsection*{Organization of the article:}
In Section \ref{sec: mu-bubble}, we prove the existence of the capillary $\mu$-bubble with prescribed contact angles modelled on the unit Euclidean $\mathbb{D}^n$. In Section \ref{sec: mean_curvature-degree}, we first introduce a trace norm of the map, and then prove a scalar-mean extremality lemma. In Section \ref{sec: proof_main_thm}, we establish Theorems \ref{mainthm: scalar_mean_curvature_comparsion} and \ref{thm: generalized_scalar_mean_comparsion}. In Section \ref{sec: Lip_scalar_mean_rigdity}, to further address the rigidity results in Corollaries \ref{corollary: spherical_radii} and \ref{corollary: Best_fill_in}, we first introduce a trace function on oriented vector spaces, followed by proving a Lipschitz scalar-mean rigidity theorem. In Appendix \ref{sec: capillarysurface}, we set up the capillary $\mu$-bubble under general conditions, and then calculate the first and second variations of the capillary $\mu$ bubble with full details. In Appendix \ref{sec: MP}, we provide details that the capillary $\mu$-bubble has no barriers that have been used in Section \ref{sec: mu-bubble}. Finally, in Appendix \ref{sec: Shi-Tam}, we briefly review the Shi-Tam inequality and its extension.
\subsection*{Acknowledgement:} The authors would like to express their gratitude to Professors Otis Chodosh, Yuguang Shi, Zhizhang Xie, Guoliang Yu, and Xin Zhou for their insightful discussions and comments. We also wish to thank Dr. Xiaoxiang Chai for bringing his work in this area \cites{ChaiWan24,ChaiWang23} to our attention, and Dr. Yujie Wu for sharing her preprint \cite{Wu_capillarysurfaces} via email.

\section{Preparations on the capillary $\mu$-bubble} \label{sec: mu-bubble}
In this section, we will first set up the minimization problem of the capillary $\mu$-bubble $\mathcal{A}_c$ on a compact manifold with nonempty boundary, and then we will prove an existence lemma of the minimizers of $\mathcal{A}_c$ in our context.

Suppose that $(M^n, \partial M, g)$ is a compact Riemannian manifold with non-empty boundary $S = \partial M$.
Consider a domain $\Omega \subset M$ and denote $\partial \Omega \cap \mathring{M} = Y$, $\bar{Y} \cap \partial M = Z$ (see Figure \ref{setup} in Appendix \ref{sec: capillarysurface} or Figure \ref{sec: setup} in this section for details).
Let $\mu_\partial$ be a smooth function on $\partial M$ with $|\mu_{\partial }| \leq 1$, then we define
\begin{equation} \label{eq: minimal_surface_angle}
    \mathcal{A}_c(\Omega) = \mathcal{H}_g^{n-1} (\partial^* \Omega \cap \mathring{M}) - \int_{\partial^*\Omega \cap S} \mu_\partial \, \mr d\mathcal{H}_g^{n-1}.
\end{equation}
for any $\Omega$ in $\mathcal{C}$, where 
\[\mathcal{C} = \left\{\text{Caccioppoli sets } \Omega \subset M \text{ with certain given  properties} \right\}.\]
\begin{definition}\ 
\begin{enumerate}
    \item A domain $\Omega \subset M$ is said to be {\em $\mc A_c$ stationary} if it is a critical point of $\mathcal{A}_c$ {among the class $\mathcal{C}$}.
    \item A domain $\Omega \subset M$ is said to be {\em an $\mathcal{A}_c$ capillary stable bubble} if $\Omega$ is a minimizer of $\mathcal{A}_c$ among the class $\mathcal{C}$. 
\end{enumerate}
\end{definition}

For the definition of the capillary $\mu$-bubble and related calculations in a general context, please refer to Appendix \ref{sec: capillarysurface}.
To provide motivation for the reader, we present a classical example for the standard unit ball $\mathbb{D} \subset \mathbb{R}^n$, which will serve as a model for the main theorems.

\begin{example}
    Suppose that $\mathbb{D}^n \subset \mathbb{R}^n$ is the standard unit ball with boundary unit sphere $\mathbb{S}^{n-1}$. Consider the spherical coordinates of $\mathbb{S}^{n-1}$ as follows.
    \[\big(\Theta\cdot \sin(\Psi), -\cos(\Psi)), \Psi \in [0, \pi].\]
    Here, $\Theta$ is the coordinate of $\mathbb{S}^{n-2}\subset\R^{n-1}$. Let $L_{\Psi_0}$ be the slice $\Psi=\Psi_0$. A direct calculation shows that the angle between $L_{\Psi_0}$ and the boundary $\mathbb{S}^{n-1}$ is equal to $\Psi_0$. See Figure III in the Appendix \ref{fig:mu-bubble}.
    In this case, if we consider $\mu_\partial = cos(\Psi)$, then any set $\{\Psi\leq\Psi_0\}$, namely the subset of $\mathbb D$ below $L_{\Psi_0}$, is stationary and stable for any $\Psi_0 \in [0, \pi]$.  In this paper, for any point $x \in \mathbb{S}^{n-1}$, we can denote by $\Psi(x)$ the angle between the $\Psi$ slices determined by $x$ and the boundary $\mathbb{S}^{n-1}$.
\end{example}

Suppose that $\{\pm p\}$ are the north and south poles of $\sph^{n-1}$. 
Considering $\Psi$ as a smooth function on $\sph^{n-1}\setminus\{{\pm p}\}$, the metric $g_{\sph^{n-1}}$ is indeed a warped product metric
$$g_{\sph^{n-1}}=\mr d\Psi^2+\big(\sin(\Psi)\big)^2g_{\sph^{n-2}}.$$
It directly deduces the following metric property of the projection map.
\begin{lemma}\label{lemma:dP}
Suppose that $\{\pm p\}$ are the north and south poles of $\mb S^{n-1}$. If $P_{n-1}: \mathbb{S}^{n-1}-\{\pm p\} \rightarrow \sph^{n-2}$ is the projection map defined by
    \[P_{n-1}: (\Theta\cdot \sin(\Psi), -\cos(\Psi) ) \mapsto \Theta,\]
    then $$\|\mr dP_{n-1}\|=\frac{1}{\sin(\Psi)}$$
    for any point in $\mathbb{S}^{n-1}$.
\end{lemma}

The minimization problem of $\mc A_c$ may have a trivial solution, i.e. the minimizer is an empty set. In the following, we now consider the following constrained minimization problem that always has a non-empty solution.

\begin{lemma}\label{lemma: existence_mu_bubble}
    With the notations above. Suppose that $(M^n, \partial M, g)$ is a smooth compact Riemannian manifold with nonempty boundary $\partial M$. If

    \begin{itemize}
        \item $S:=\partial M$ has positive mean curvature $H_{S} >0$;
        \item $F: \partial M \rightarrow \mathbb{S}^{n-1}$ is a smooth map with $\deg(F) \neq 0$, and $F$ maps a small, smooth geodesic ball $B_1 \subset S $(resp. $B_2 \subset S)$ to a very small neighborhood of the south pole $-p \in \mathbb{S}^{n-1}$ (resp. north pole $+p$) of $\mathbb{S}^{n-1}$;

        \item In line (\ref{eq: minimal_surface_angle}), we set $\mu_{\partial}(s) = (\cos(\Psi(F(s)))$ for any $s \in S$;

        \item $n = 2, 3, 4$,  
    \end{itemize}
    then there exists a smooth, stable capillary $\mu$-bubble $\Omega $ in $M$ for which the boundary $Y := \partial \Omega \cap \mathring{M} $ satisfies the following properties:
    \begin{enumerate}
        \item \label{item: first_variation} First variation: $H_{Y} = 0$ on $Y$ and $J(z) = \Psi(F(z))$ for any $z \in  Z = \partial Y$ where $J(z)$ is the contact angle between $Y$ and $S$ at the intersection point $z \in Z= \partial Y$;
        \item \label{item: second_variation}{Stability: for any $\varphi\in C^\infty(Y)$,
        \begin{align*}    
        \mc Q(\varphi,\varphi) :=&\int_Y |\nabla \varphi|^2  - \big(\Ric_g(\nu_Y, \nu_Y) + \|A_Y\|^2  \big) \varphi^2 \,\mr d\mathcal{H}_g^{n-1}  \\
       + & \int_Z \big(  H_Z - \frac{H_S}{\sin(J)}  +\frac{1}{\sin(J)} \langle \nabla \Psi, \mr dF(\n) \rangle \big)\varphi^2  \,\mr d\mathcal{H}_g^{n-2} \geq 0, 
   \end{align*}
where $\n$ is the unit, upward normal vector field of $Z$ in $S$; $\nu_Y$ is the outward unit normal of $Y$. Here, we will write $\nabla J = \nabla \Psi |_{F(Z)}$ for notation abuse whenever it is no confusion. }
   \item  \label{item: nonzero_degree}Preserve non-zero degree: there exists a connected component of $Y$ still denoted by $Y$, and a smooth map $$F_{n-2}: \partial Y \rightarrow \mathbb{S}^{n-2}$$ 
   with $\deg(F_{n-2}) \neq 0$. 
    \end{enumerate}
    \end{lemma}

\begin{proof}
    We mainly focus on the proof of the existence of a stable capillary $\mu$-bubble. The variation formulas in item (\ref{item: first_variation}) and the stability in (\ref{item: second_variation}) follow from the calculations in the Appendix \ref{sec: capillarysurface} and Lemma \ref{lem:vector extension}; the argument of non-zero degree of the map $F_{n-2}$ follows from \cite{CWXZ_Llarull_4}*{Lemma 3.2}.
    
    Now we set \[\mathcal{C} = \big\{\text{Caccioppoli sets } \Omega \subset M \text{ such that } \partial^*\big(\partial^* \Omega \cap \mathring{ M}\big) \subset \partial M \setminus\big( B_1 \cup B_2 \big) \text{ and } B_1 \subset  \Omega 
 \big\}.\]
 
 	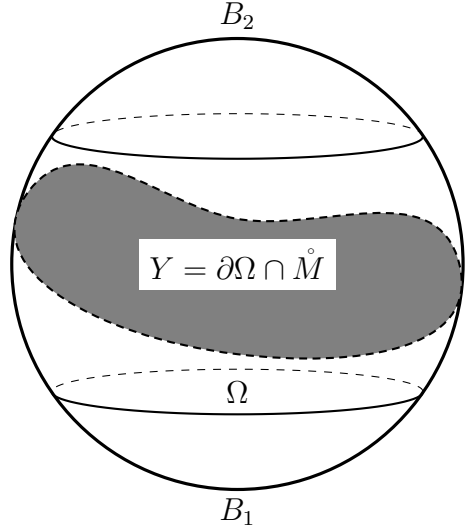
\begin{figure}[ht]
		\begin{tikzpicture}[use Hobby shortcut,scale=1]
			\draw[very thick] (3,0) arc(0:360:3);
			\draw[dashed,domain=0:180,smooth] plot({2.47*cos(\x)},{0.3*sin(\x)+1.7});
			\draw[thick,domain=180:360,smooth] plot({2.47*cos(\x)},{0.3*sin(\x)+1.7});
			\draw[dashed,domain=0:180,smooth] plot({2.47*cos(\x)},{0.3*sin(\x)-1.7});
			\draw[thick,domain=180:360,smooth] plot({2.47*cos(\x)},{0.3*sin(\x)-1.7});
        \node at (0,-1.7){$\Omega$};
			\node at (0,3.3){$B_2$};
			\node at (0,-3.3){$B_1$};
        \fill[gray,opacity=0.1,closed] (-2.75,1) .. (0,.6) .. (2.95,0) .. (0,-1.2);
        \node at (0,0)[fill=white]{$Y = \partial \Omega \cap \mathring{M}$};
			\draw[dash between=0 and 0.45,
     dash between style=dashed,closed,thick] (-2.75,1) .. (0,.6) .. (2.95,0) .. (0,-1.2);
		\end{tikzpicture}
  \caption{ $\mu$ -bubble setup} \label{sec: setup}
	\end{figure} 
Since $(M, \partial M, g)$ is a smooth, compact Riemannian manifold, we obtain 
        \begin{equation*}
           I :=  \inf_{\Omega \in \mathcal{C}}  \mathcal{A}_c(\Omega) \text{ exists}.
        \end{equation*}
Next we assume that $\{\Omega_k\}_{k=1}^\infty \subset M$ is a minimizing sequence of $\mathcal{A}_c$ such that
\[\lim_{k \rightarrow \infty} \mathcal{A}_c(\Omega_k) = I.\]
Consequently, by the definition of the  minimizing sequence of $\{\Omega_k\}_{k=1}^\infty$, we obtain that
\[\mathcal{H}_g^{n-1}(\partial^*\Omega_k) \leq I + 1+ \mathcal{H}_g^{n-1}(\partial M)\]
for large $k$.
Note that the minimization problem in the context has obstacles in the following two aspects: 

\begin{enumerate}
    \item \label{item: attach_S} The interior of $\partial \Omega_k \cap \mathring{M}$ may attach the boundary $S$,

    \item \label{item: attach_corner}  $\partial Y_k$ may move closer and closer to the set $\partial B_1$ or $\partial B_2$ as $k \to \infty$. Here, $Y_k =  \partial \Omega_k \cap \mathring{M} $.
\end{enumerate}

Note that since $H_S > 0$ on the boundary $S$, the case (\ref{item: attach_S}) will be ruled out by the strong maximum principle in the interior (see \cite{Li_Zhou_MP}*{Theorem 1.2}). Moreover,  case (\ref{item: attach_corner}) is prevented from the maximum principle on the boundary (see \cite{Wu_capillarysurfaces}*{Step 4 \& 5 in the proof of Theorem 1.3 on pages 5-6})\footnote{The argument {applies to} all dimensions in \cite{Wu_capillarysurfaces}.}. For the readers' convenience, we will provide details as Claim \ref{claim: mp} in Appendix \ref{sec: MP}.

 Hence, the minimization problem of $\mathcal{A}_c$ does not have a barrier. Therefore, by a recent regularity theorem on the capillary $\mu$-bubble in \cite{chodosh2024_improvedregularity}*{Theorem 1.1} for $n\leq 4$, we conclude that $Y_k: = \partial \Omega_k$ converges to a smooth hypersurface $Y \subset M$ such that
\begin{itemize}
    \item $\displaystyle \mathring{Y} \subset \mathring{M}$;
    \item $Z:= Y \cap S$ is a smooth nonempty hypersurface in $S$.
\end{itemize}
Hence, we finished the proof and note that we only used the dimension assumption on the regularity.
\end{proof}

\section{Scalar-mean extremality} 
\label{sec: mean_curvature-degree}

In this section, we will first prove the scalar-mean extremality theorem that is a weaker version of Theorem \ref{thm: generalized_scalar_mean_comparsion}. 

\vspace{1mm}
Suppose that $(M^n, \partial M, g)$ is a smooth, compact Riemannian manifold with nonempty boundary $\partial M$ and $F: (\partial M, g) \to (\mathbb{S}^{n-1}, g_{\sph^{n-1}})$ is a smooth map. Recall that the trace norm of $\mr dF$ at any point $x$ in $\partial M$ is defined by
\begin{equation}\label{def: trace_norm}
    \|\mr dF\|_{\tr}(x)\coloneqq \sup_{\{u_i\},\{v_i\}}\sum_{i=1}^{n-1}|\langle \mr dF_x(u_i),v_i\rangle|.
\end{equation}
Here, the supremum is taken over both all orthonormal bases $\{u_i\}_{1\leq i\leq n}$ of $T_x \partial M$ and the orthonormal basis $\{v_i\}_{1\leq i\leq n}$ of $T_{F(x)}\sph^{n}$. We may also write $\|\mr dF\|_{\tr,g}$ to emphasize that the trace norm is taken with respect to the Riemannian metric $g$ on $\partial M$.

\begin{lemma} \label{lemma: perturbation}
    Suppose that $(M^n,g)$ is a smooth closed Riemannian manifold. If 
    $$ F \colon (M,g)\to(\sph^{n},g_{\sph^{n}})$$ 
    is a smooth map such that \begin{equation}\label{eq:delta}
		\|\mr d F\|_{\tr} < A
	\end{equation}
 for some smooth function $A$ on $M$, then there exists a smooth map $$F': (M,g)\to(\sph^{n},g_{\sph^{n}})$$ with the following properties.
\begin{itemize}
    \item There exit small open balls $B_1, B_2$ in $\partial M$ such that $$F'(B_1) = \{-p\}, F'(B_2) = \{+p\}$$ where $\{\pm p\}$ are the north and south poles of $\mathbb{S}^n$, and
    \item $$\|\mr d F\|_{\tr} < A, \ \deg(F') =\deg(F).$$
\end{itemize}
\end{lemma}
\begin{proof}
    Since  $\|\mr dF\|_{\tr} < A$ by our assumption and $(M, g)$ is a smooth, closed Riemannian manifold, there exists a positive constant $\delta > 0$ such that 
    $$\|\mr dF\|_{\tr} < (1-\delta)A.$$
    Consequently, by scaling, there exists a smooth map
    $$F_\delta : (M, g) \to \bigg(\mathbb{S}^n(\frac{1}{1-\delta}),\  g_{\mathbb{S}^n(\frac{1}{1-\delta})}\bigg)$$
    such that 
   $$\|\mr d F_{\delta}\|_{\tr} < A,\ 
    \deg(F_{\delta}) =\deg(F).$$

Next it is straightforward to construct a map 
$$\pi: \bigg(\mathbb{S}^n(\frac{1}{1-\delta}), g_{\mathbb{S}^n(\frac{1}{1-\delta})}\bigg) \to (\sph^{n},g_{\sph^{n}})$$ by collapsing the small south and north spherical caps of $ \displaystyle \big(\mathbb{S}^n(\frac{1}{1-\delta}), g_{\mathbb{S}^n(\frac{1}{1-\delta})}\big)$ to the south pole $-p$ and north pole $+p$ of $(\sph^{n},g_{\sph^{n}})$ with
  \begin{itemize}
    \item $\|\mr d \pi\| \leq 1$, where $\|\mr d \pi\|$ stands for the $l^\infty$-matrix norm of $d\pi$, and
    \item $\deg(\pi) \neq 0$.
\end{itemize}
Consequently, $F' := \pi \circ F_\delta$ is the map as required.
\end{proof}

Using the Existence Lemma \ref{lemma: existence_mu_bubble} for the capillary $\mu$-bubble from Section \ref{sec: mu-bubble} and the Perturbation Lemma \ref{lemma: perturbation} from Section \ref{sec: mean_curvature-degree}, we can establish the following extremality theorem.

\vspace{1mm}
\begin{proposition}\label{prop:weak}
Suppose that $(M^n, \partial M, g)$ is a smooth compact Riemannian manifold with  non-empty mean convex boundary $\partial M$ and non-negative scalar curvature $\Sc_g \geq 0$ in $M$. If $ F \colon (\partial M,g|_{\partial M})\to(\sph^{n-1},g_{\sph^{n-1}})$ is a smooth map such that \begin{equation}\label{eq:delta_2}
		H_{\partial M}\geq \|\mr dF\|_{\tr}+\delta \ \text{ on } \partial M
	\end{equation}
for some positive constant $\delta > 0$ and $n=2, 3, 4$, then $\deg(F)=0$.
\end{proposition}
\begin{proof}
Note that the proposition holds for $n=2$ due to the Gauss–Bonnet formula for a compact surface $M^2$ with a nonempty boundary $\partial M$. Specifically, we recall that
\begin{equation*}
    \frac 1 2\int_M \Sc_g d\mathcal{H}_g^2 + \int_{\partial M } k_{\partial M} \, \mr d\mathcal{H}_g^1 =  2\pi\chi(M) \leq 2\pi. 
\end{equation*}
Here, $k_{\partial M}$ is the geodesic curvature of $\partial M$ in $M$,which also coincides with the mean curvature $H_{\partial M}$ of $\partial M $ in $M$. Moreover, under our assumption that $\deg(F) \neq 0$, we obtain
\[
    \frac 1 2\int_M \Sc_g\,\mr d\mathcal{H}_g^2 + \int_{\partial M } k_{\partial M } \,\mr d\mathcal{H}_g^1 \geq 2\pi + \delta \cdot \mathcal{H}_g^1(\partial M ), \quad  \delta > 0 .\]
This implies that $\mathcal{H}_g^1(\partial M ) = 0$,
contradicting the assumption that the boundary is nonempty.

\vspace{1mm}

Next, we employ dimension reduction to extend the argument to manifolds of higher dimensions. Suppose the statement holds for manifolds of dimension \( n-1 \), where \( n \geq 3 \). We now demonstrate that it also holds for manifolds of dimension \( n \geq 3 \). \textbf{For consistency in notation, we will set \( m = n-1 \) when performing the conformal change below}.

We now assume that for manifolds of dimension $n, n\geq 3$, the following holds:
\[H_g\geq \|\mr dF\|_{\tr}+\delta \text{ and } \deg(F) \neq 0\]
for some positive constant $\delta > 0$. Then   Lemma \ref{lemma: perturbation} implies that there exists a smooth map $$F_{n-1}: (\partial M, g_{\partial M}) \rightarrow (\sph^{n-1}, g_{\sph^{n-1}})$$
such that
\begin{itemize}
    \item there are small open balls $B_1, B_2$ in $\partial M$ with $$F_{n-1}(B_1) = \{-p\} \text{ and } F_{n-1}(B_2) = \{p\};$$

    \item $\deg(F_{n-1}) \neq 0$;

    \item $H \geq \|\mr dF_{n-1}\|_{\tr} +  \delta$ for a new small $\delta > 0$. 
\end{itemize}

Thus, the smooth compact Riemannian manifold $(M^n, \partial M, g)$, together with $F_{n-1}$, satisfies the conditions in Lemma \ref{lemma: existence_mu_bubble}. Consequently,there exists a smooth hypersurface $$(Y^m, \partial Y, g_{Y}) \subset (M^n, \partial M , g)$$ with properties (\ref{item: first_variation}), (\ref{item: second_variation}) and (\ref{item: nonzero_degree}) as stated in  Lemma \ref{lemma: existence_mu_bubble}. In particular, we obtain that, for any non-zero $\varphi\in C^\infty(Y)$, we obtain
\begin{align}    \label{stable}
        0<&\int_Y |\nabla \varphi|^2  - \big(\Ric_g(\nu_Y, \nu_Y) + \|A_Y\|^2  \big) \varphi^2 \,\mr d\mathcal{H}_g^{m} +  \\
       \label{stable2} & \int_Z \big(  H_Z - \frac{H_S}{\sin(J)}  +\frac{1}{\sin(J)} \langle \nabla \Psi, \mr dF(\n) \rangle \big)\varphi^2  \,\mr d\mathcal{H}_g^{m-1}.
   \end{align}
Here, let us first write for short
   \[B := H_Z - \frac{H_S}{\sin(J)}  +\frac{1}{\sin(J)} \langle \nabla \Psi, \mr dF(\n) \rangle=H_Z - \frac{H_S}{\sin(J)} - \frac{1}{\sin^2(J)}\frac{\partial \mu_\partial}{\partial \n}.\] 
Moreover, by the Gauss equation (see Schoen-Yau trick in \cite{SY_pmt_1}*{page. 54}) and our assumption of nonnegative scalar curvature on $M$, we have
\[-(\Ric(\nu_Y, \nu_Y)+ \|A\|^2) = -\frac{1}{2}\Sc_g +\frac 1 2 \Sc_{g_Y} -\frac 1 2  \|A\|^2\leq \frac 1 2\Sc_{g_Y}.\]
Thus, the equations in lines \eqref{stable} and \eqref{stable2} together imply that
\begin{equation}\label{eq:>0}
           \int_Y (|\nabla \varphi|^2  +\frac 1 2 \Sc_{g_Y}\varphi^2)\,\mr d\mathcal{H}_g^{m}  
       + \int_Z B\varphi^2  \,\mr d\mathcal{H}_g^{m-1} > 0, \text{for any } \varphi \in C^\infty(Y).
\end{equation}

\vspace{2mm}
\noindent\textbf{Case I: $n=3$.} In this case, we obtain $m=\dim Y=2$, so $Y$ is a connected oriented surface with nonempty boundary. By the uniformization theorem, $Y$ is homeomorphic to $\Sigma^{g,k}$, the genus-$g$ surface with $k$ discs removed ($k\geq 1$). Therefore, we have $$\chi(Y)=2-2g-k\leq 1.$$

Moreover, following Lemma \ref{lemma:dP}, we define the boundary map as follows.
\begin{equation} \label{eq: euler}
     F_{1}= P_{1} \circ F_{2}: (\partial Y, g_{\partial Y}) \rightarrow  (\mathbb{S}^{1}, g_{\mathbb{S}^{1}})
\end{equation}
where $P_{1}$ is the projection map to the equator $\mb S^{2}$ of $\mb S^{2}$. By our assumption on $H_S$ and a direct calculation, we have
$$B\leq H_Z-\Big|\frac{\partial\theta}{\partial s}\Big|,$$
where $\theta$ is the angle parameter of the equator $\mb S^{n-2}=\mb S^1$ and $s$ is the arc length parameter of $Z=\partial Y$.
Therefore, by setting $\varphi=1\in C^\infty(Y)$ in line \eqref{eq:>0}, we obtain
\begin{align*}
    0<&\int_Y\frac 1 2 \Sc_{g_Y}\,\mr d\mathcal{H}_g^{2}  
       + \int_Z(H_Z-\Big|\frac{\partial\theta}{\partial s}\Big|)  \,\mr d\mathcal{H}_g^{1}\\
       \leq &~2\pi\chi(Y)-2\pi.
\end{align*}
The last inequality follows from the Gauss–Bonnet formula. Thus, $\chi(Y)>1$ and it contradicts with the fact that $\chi(Y)>1$ in line \eqref{eq: euler}

\vspace{2mm}
\noindent\textbf{Case II: $n\geq 4$.} Note that our theorem holds for 
$n=4$ due to the insufficient regularity of the capillary $\mu$-bubble. However, the dimension reduction argument is applicable for all dimensions. Therefore, we perform this part for the general case where $n \geq 4$.

\vspace{1mm}
We first consider the following eigenvalue Neumann boundary problem on $Y$,
\begin{equation}\label{eq:Lboundary}
\begin{cases}
        \displaystyle \mathcal Lf : =-\Delta f+\frac{m-2}{4(m-1)}\Sc_{g_Y}f = \kappa f, \text{ in } Y^m,\\
        \displaystyle \frac{\partial f}{\partial\nu_Z}=-\frac{m-2}{2(m-1)}Bf \text{ on } Z=\partial Y,
\end{cases}
\end{equation}
where $\nu_Z$ is the unit outer normal vector field of $Z=\partial Y$ in $Y$.

Then, we claim that the bottom eigenvalue $\kappa$ of the Neumann boundary problem is positive. To prove this, assume for contradiction that $\kappa\leq 0$, then there is a non-zero $f\in C^\infty (Y)$ such that
$$\begin{cases}
        \displaystyle \mathcal Lf=\kappa f, \text{ in } Y^m,\\
        \displaystyle \frac{\partial f}{\partial\nu_Z}=-\frac{m-2}{2(m-1)}Bf \text{ on }Z=\partial Y.
\end{cases}$$
The Stokes theorem and the assumption $\kappa \leq 0$ imply that
\begin{align*}
    0 \geq \langle\mathcal Lf,f\rangle=\int_Y|\nabla f|^2\,\mr d\mathcal{H}_g^{m}+\frac{n-2}{2(n-1)}\Big(\int_Y\frac 1 2\Sc_{g_Y}f^2\,\mr d\mathcal{H}_g^{m}+\int_Z Bf^2\,\mr d\mathcal{H}_g^{m-1}
    \Big).
\end{align*}
Moreover, from the equation in line \eqref{eq:>0}, we obtain
$$-\int_Y|\nabla f|^2\,\mr d\mathcal{H}_g^{m}<\int_Y\frac 1 2\Sc_{g_Y}f^2\,\mr d\mathcal{H}_g^{m}+\int_Z Bf^2\,\mr d\mathcal{H}_g^{m-1}\leq -\frac{2(m-1)}{m-2}\int_Y|\nabla f|^2\,\mr d\mathcal{H}_g^{m}.$$
This leads to the inequality
$$\frac{m}{m-2}\int_Y|\nabla f|^2\,\mr d\mathcal{H}_g^{m}<0.$$
which is a contradiction. Hence, our assumption that $\kappa \leq 0$ must be false, and we conclude that the bottom eigenvalue 
$\kappa$ is positive. This proves our claim.
This leads to a contradiction and hence proves our claim.
Therefore, the classical elliptic theory implies that there exists $\kappa>0$ and a  positive function $f\in C^\infty(Y)$ solving 
$$\begin{cases}
        \displaystyle \mathcal Lf=\kappa f, \text{ in } Y^m,\\
        \displaystyle \frac{\partial f}{\partial\nu_Z}=-\frac{m-2}{2(m-1)}Bf \text{ on } Z=\partial Y.
\end{cases}$$

\vspace{1mm}

Next we consider the conformal metric on $Y$ as follows,
 $$(Y^m, \partial Y, g_f)=(Y, \partial Y, f^{\frac{4}{m-2}}g_Y).$$ 
We denote $\Sc_{g_f}$ by the scalar curvature in $(Y, \partial Y, g_f)$ and $H_{Z, g_f}$ by the mean curvature of $Z=\partial Y$  on $(Y, \partial Y, g_f)$. Recall that
the scalar curvature and the mean curvature under the conformal deformation are given by
\begin{equation}\label{eq:scalar_f}
    \Sc_{g_{f}}=\frac{4(m-1)}{m-2}f^{-\frac{m+2}{m-2}}\left(-\Delta f + \frac{m-2}{4(m-1)}\Sc_{g_Y} f \right),
\end{equation}
and
\begin{equation}\label{eq:mean_f}
    H_{Z, g_{f}} =  f^{-\frac{2}{m-2}}\left(H_{Z, g}  + \frac{2(m-1)}{m-2}\frac{1}{ f}\frac{\partial f}{\partial \nu_Z}\right).
\end{equation}

Then we consider the boundary map
\begin{equation}
     F_{n-2}= P_{n-2} \circ F_{n-1}: (\partial Y, g_{\partial Y}) \rightarrow  (\mathbb{S}^{n-2}, g_{\mathbb{S}^{n-2}}).
\end{equation}
By the definition of the trace norm in line (\ref{def: trace_norm}), a straightforward calculation yields
\begin{align*}
     \|\mr d F_{n-2}\|_{\tr,g_f} &=f^{-\frac{2}{m-2}}\|\mr d P_{n-2}\circ \mr d(F_{n-1}|_Z)\|_{\tr,g_Z}\\
     &=\frac{1}{\sin(J)}f^{-\frac{2}{m-2}}\|\mathbb P_{n-2}\circ \mr d(F_{n-1}|_Z)\|_{\tr,g_Z}.
\end{align*}
Here, we have $\mr d P_{n-2} = \frac{1}{\sin(J)}\mathbb P_{n-2}$ with $\mathbb P_{n-2}$ the orthogonal projection from $T{\mathbb{S}^{n-1}}$ onto the orthogonal complement of $\nabla \Psi$ in $T{\mathbb{S}^{n-1}}$. 

Recall that the definition of the trace norm yields that for any $z\in Z$,
$$\|\mathbb P_{n-2}\circ \mr d(F_{n-2}|_Z)\|_{\tr_{g_Z}} (z)=\sup_{\{u_i\},\{v_i\}}\sum_{i=1}^{n-2}|\langle \big(\mathbb P_{n-2}\circ \mr d(F_{n-1}|_Z)\big)(u_i),v_i\rangle|,$$
where the supremum is taken over all orthonormal basis $\{u_i\}_{1\leq i\leq n-2}$ of $T_z(\partial M)$ and orthonormal vectors $\{v_i\}_{1\leq i\leq n-2}$ of $T_{f(z)}\sph^{n-1}$. Since $\mathbb P_{n-2}$ is self-adjoint, then we have
\begin{align*}
    \|\mathbb P_{n-2}
\circ \mr d(F_{n-1}|_Z)\|_{\tr_{g_Z}} (z)&=\sup_{\{u_i\},\{v_i\}}\sum_{i=1}^{n-2}|\langle  \mr d(F_{n-1}|_Z)(u_i),\mathbb P_{n-2}v_i\rangle|(z)\\
    &=\sup_{\{u_i\},\{w_i\}}\sum_{i=1}^{n-2}|\langle\mr d(F_{n-1}|_Z)(u_i),w_i\rangle|(z)
\end{align*}
where the second supremum is taken over all orthonormal basis $\{u_i\}_{1\leq i\leq n-2}$ of $T_z(\partial M)$ and orthonormal basis $\{w_i\}_{1\leq i\leq n-2}$ of $\mathbb P_{n-2} T_{f(z)}\sph^{n-1}$. Note that  $ \displaystyle \big\{\nabla\Psi, w_1,\ldots,w_{n-2}\big\}$ forms an orthonormal basis of $T_{f(z)}\sph^{n-1}$ and $\displaystyle \big\{ \n, u_1,\ldots,u_{n-2}\big\}$ forms an orthonormal basis of $T_z(\partial M)$, we have
\begin{align*}
    &\ \ \ \ \|P^\perp_{n-2}\circ \mr d(F_{n-1}|_Z)\|_{\tr,g_Z}+|\langle \n,\nabla J \rangle|\\
    &=\|P^\perp_{n-2}\circ\mr d(F_{n-1}|_Z)\|_{\tr,g_Z}+|\langle \mr dF_{n-1}(\n),\nabla \Psi\rangle|\\
    &\leq \|\mr dF_{n-1}\|_{\tr,g_Z}.
\end{align*}
Hence, by our assumption on $H_S$ in line \eqref{eq:delta} and the equation in line \eqref{eq:mean_f}, we obtain
\begin{align*}
	H_{Z, g_f}=&f^{-\frac{2}{m-2}}\left(H_{Z, g}  + \frac{2(m-1)}{m-2}\frac{1}{ f}\frac{\partial f}{\partial \nu_Z}\right)\\
    =&\frac{1}{\sin(J)}f^{-\frac{2}{m-2}}\big(H_S-\langle \n,\nabla J \rangle\big)\\
	\geq&\frac{1}{\sin(J)}f^{-\frac{2}{m-2}}\bigg(\|\mr dF_{n-1}\|_{\tr,g_S}+\delta-|\langle \n,\nabla J\rangle|\bigg)\\
	\geq& \|\mr dF_{n-2}\|_{\tr,g_f}+\delta\cdot \frac{1}{\sin(J)}f^{-\frac{2}{m-2}}.
\end{align*}
Since $F_{n-1}(\partial Y)$ stays away from the poles and $f$ is strictly positive on $Y$, we get 
$$\widetilde\delta\coloneqq \delta\cdot \inf_{Z}\frac{1}{\sin(J)}f^{-\frac{2}{m-2}}>0.$$

To summarize, we obtain a smooth compact Riemannian manifold $(Y^{n-1}, \partial Y, g_f)$  of dimension $m=n-1$ that satisfies the following.
\begin{itemize}
    \item The scalar curvature of $(Y^{n-1}, \partial Y, g_f)$ is given by
\begin{align}
      \label{eq:Sc>=0}\Sc_{g_{f}}&=\frac{4(m-1)}{m-2}f^{-\frac{m+2}{m-2}}\left(-\Delta f + \frac{m-2}{4(m-1)}\Sc_{g_Y} f\right)=\frac{4(m-1)}{m-2}f^{-\frac{m+2}{m-2}}\mathcal Lf  \\
      \nonumber &= \frac{4(m-1)}{m-2}\kappa f^{-\frac{4}{m-2}}>0.
\end{align}
     \item The mean curvature of $Z=\partial Y$ in of $(Y^{n-1}, \partial Y, g_f)$ is given by
    \begin{align}
	\label{eq:H>}H_{Z, g_f}&=f^{-\frac{2}{m-2}}\left(H_{Z, g}  + \frac{2(m-1)}{m-2}\frac{1}{ f}\frac{\partial f}{\partial \nu_Z}\right)\\
    \nonumber &\geq \|\mr dF_{n-2}\|_{\tr,g_f}+\widetilde\delta
\end{align}
for some positive constant $\widetilde\delta>0$, where
$$F_{n-2}: (\partial Y, g_{f}|_{\partial M}) \rightarrow (\sph^{n-2}, g_{\sph^{n-2}})$$
has non-zero degree.
\end{itemize}
This contradicts the assumption that the statement holds for manifolds of dimension $(n-1)$. Hence, we conclude that $\deg(F_{n-1}) = 0$. This finishes the proof.
\end{proof}

\begin{remark}
  The \emph{dimension reduction} argument for mean curvature works the same as the Schoen-Yau dimension reduction for scalar curvature if one can improve the regularity of capillary $\mu$-bubble generically for the manifold of higher dimensions $n \geq 5$. 
\end{remark}

\section{The proof of the main theorems}
\label{sec: proof_main_thm}

In this section, we will prove Theorem \ref{mainthm: scalar_mean_curvature_comparsion} and Theorem \ref{thm: generalized_scalar_mean_comparsion}.

\subsection{Scalar-mean rigidity theorem}
In this subsection, we will prove the scalar-mean rigidity theorem \ref{mainthm: scalar_mean_curvature_comparsion}. Here, we shall state the theorem for the reader's convenience.

\begin{theorem} \label{mainthm: geometric_verison}
       Suppose that $(M^n, \partial M, g), n=2, 3, 4$ is a smooth, compact Riemannian manifold with nonnegative scalar curvature $\Sc_g \geq 0$ and uniformly positive mean curvature  $H_{\partial M} \geq n-1$. If $F: (\partial M, g_{\partial M} )\rightarrow (\mathbb{S}^{n-1}, g_{\mathbb{S}^{n-1}})$ is a distance non-increasing  map of $\deg(F)  \neq 0$, then $F$ is an isometry, and $(M, g)$ is isometric to $(\mathbb{D}^n, g_{\mathbb{D}^n})$.
\end{theorem}
\begin{proof}
   The statement holds for $n=2$ due to the Gauss--Bonnet formula on a smooth compact manifold with non-empty boundary. We will focus on the case for $n = 3, 4$ in the proof.
   
\textbf{Claim A}: Under the assumption of Theorem \ref{mainthm: geometric_verison}, we have
\begin{equation}\label{eq:extremal}
		\Sc_g=0 \text{ on } M; ~H_{\partial M}=\|\mr dF\|_{\tr}\text{ and } \|\mr dF\|_{\tr}=n-1 \text{ on } \partial M.
	\end{equation}

Let us argue by contradiction. Suppose that at least one of these three equalities in line (\ref{eq:extremal}) fails at some point in $M$, then let us consider the following Neumann eigenvalue problem on $(M^n, \partial M, g)$,
	\begin{equation} \label{eq: NB_H}
		\begin{cases}
			-\displaystyle \Delta \varphi +\frac{n-2}{4(n-1)} \Sc_g\varphi = \lambda \varphi,\\
			\displaystyle\frac{\partial \varphi}{\partial \nu}=-\frac{n-2}{2(n-1)}(H-\|\mr dF\|_{\tr})\varphi,
		\end{cases}
	\end{equation}
where $\nu$ is the unit outer normal vector field  of $\partial M$.
The Green formula shows that
\begin{align}
&\ \ \lambda\int_M \varphi^2 \,\mr d\mathcal{H}_g^{n}\\
=&\int_M |\nabla \varphi|^2  +\frac{n-2}{4(n-1)}\Sc_g\varphi^2\,\mr d\mathcal{H}_g^{n} - \int_{\partial M} \varphi \frac{\partial \varphi}{\partial \nu} \,\mr d\mathcal{H}_g^{n-1} \nonumber\\
  =&\int_M |\nabla \varphi|^2\,\mr d\mathcal{H}_g^{n} +\frac{n-2}{4(n-1)}\Sc_g\varphi^2\,\mr d\mathcal{H}_g^{n} + \frac{n-2}{2(n-1)}\int_{\partial M } (H - \|\mr dF\|_{\tr}) \varphi^2 \,\mr d\mathcal{H}_g^{n-1}\nonumber\\
=&\int_M |\nabla \varphi|^2\,\mr d\mathcal{H}_g^{n} +\frac{n-2}{4(n-1)}\Sc_g\varphi^2\, \mr d\mathcal{H}_g^{n} \label{eq:modification}\\
&+\frac{n-2}{2(n-1)}\int_{\partial M} \Big(  (H-(n-1))+((n-1)-\|\mr dF\|_\tr)\Big)\varphi^2\,\mr d\mathcal{H}_g^{n-1}.\nonumber
\end{align}
Since $F$ is distance-non-increasing, we obtain  $$\|\mr dF\|_\tr\leq n-1.$$ 
It follows that $\lambda \geq 0$. If $\lambda =0$, then $\varphi$ is a non-zero constant function. Consequently,
\begin{equation*}
    \Sc_g = 0\text{ on }M;\quad  H = n-1=\|\mr dF\|_\tr\text{ on }\partial M.
\end{equation*}
This contradicts the assumption that at least one of them fails at some point in line \ref{eq:extremal} in $M$. Consequently, the first Neumann eigenvalue $\lambda > 0$.
It implies that there exists a positive function $v$ that solves the Neumann boundary problem in line (\ref{eq: NB_H}) with a positive constant $\lambda > 0$.

\vspace{1mm}

Moreover, we consider the conformal metric on $M$ given by
$$(M, g_v) := (M, v^{\frac{4}{n-2}}g).$$
Then,
\begin{itemize}
    \item The scalar curvature $\Sc_{g_v}$ of $g_v$ on $M$ is given by
\begin{align*}
	\Sc_{g_v}=v^{-\frac{n+2}{n-2}}\Big(-\Delta v+\frac{n-1}{4(n-2)}\Sc_g v\Big)=\lambda v^{-\frac{4}{n-2}}	\geq\delta_1>0,
\end{align*}
where
$\delta_1=\lambda\inf_M v^{-\frac{4}{n-2}} > 0.$

 \item The mean curvature $H_{\partial M}$ of $\partial M$ with respect to $g_v$ is given by
\begin{align*}
	H_{g_v}=v^{-\frac{2}{n-2}}\big(H_{g}+\frac{2(n-1)}{n-2}\frac{1}{v}\frac{\partial v}{\partial \nu_S}\big)
	=v^{-\frac{2}{n-2}}\cdot\|\mr dF\|_{\tr,g_{\partial M}}.
\end{align*}

\item According to the conformal metric, we have
 $$\|\mr dF\|_{\tr,g_v}=v^{-\frac{2}{n-2}}\|\mr dF\|_{\tr,g_{\partial M}}.$$
\end{itemize}
Hence, this conformal change process increases the scalar curvature in the interior $\mathring{M}$ with possibly a sacrifice of the mean curvature on the boundary $\partial M$.

\vspace{1mm}

Next let us work on $(M^n, \partial M , g_v)$ to increase the mean curvature on the boundary using the scalar curvature. Let $\nu_{g_v}$ be the unit outer normal vector field of $\partial M$ with respect to $g_v$ and $w$ an arbitrary smooth function on $M$ such that
$$\frac{\partial w}{\partial \nu_{g_v}}=1.$$
 We further consider the perturbation conformal metric for small $\varepsilon > 0$:
\begin{equation*}
    (M^n, \partial M, g_w) = (M^n, \partial M, (1+\varepsilon w)^\frac{4}{n-2}g_v).
\end{equation*}
\begin{itemize}
    \item The scalar curvature $\Sc_{g_w}$ of $g_w$ on $(M, g_w)$ is given by
$$\Sc_{g_w}=(1+\varepsilon w)^{-\frac{n+2}{n-2}}\Big(\varepsilon\big(-\frac{4(n-1)}{n-2}\Delta w+\Sc_{g_v}w\big)+\Sc_{g_v}\Big).$$
As $\Sc_{g_v}\geq\delta_1>0$, we fix $\varepsilon$ small enough so that
$$2\geq 1+\varepsilon\inf_{M}w\geq 1,$$
and
$$\varepsilon\cdot\inf_M\big(-\frac{4(n-1)}{n-2}\Delta w+\Sc_{g_v}w\big)+\delta_1\geq\frac{\delta_1}{2}.$$
It follows that
$$\Sc_{g_w}\geq 2^{-\frac{n+2}{n-2}}\frac{\delta_1}{2}>0.$$

\item The mean curvature $H_{g_w}$ of $\partial M$ with respect to $g_w$  is given by
\begin{align*}
	H_{g_w}&=(1+\varepsilon w)^{-\frac{2}{n-2}}\big(H_{g_v}+{\frac{2(n-1)}{n-2}\frac{\varepsilon}{1+\varepsilon w}\frac{\partial w}{\partial\n_v}}\big)\\
	&=\|\mr dF\|_{\tr,g_w}+\frac{2(n-1)}{n-2}\varepsilon.
\end{align*}

\item Under the conformal metric, we have
$$\|\mr dF\|_{\tr,g_w}=(1+\varepsilon w)^{-\frac{2}{n-2}}\|\mr dF\|_{\tr,g_v}.$$
\end{itemize}

Finally, we conclude that, if  \textbf{Claim A} fails at some point in $M$, then there exists a smooth, compact
Riemannian manifold $(M, \partial M. g_w)$ coupled with a smooth map $$F:  ( \partial M, g_w)\to(\sph^{n-1},g_{\sph^{n-1}})$$ with the following properties,
\begin{itemize}
    \item  $\Sc_{g_w} >0$ in $M$;
    \item $H_{g_w}\geq \|\mr dF\|_{\tr,g_w}+\delta$ for  some constant $\delta>0$;
    \item $\deg(F) \neq 0$.
\end{itemize}
 This contradicts Proposition \ref{prop:weak}. We remark that the above computation does not have any dimension restriction. The only point that we need $\dim=3,4$ is the smoothness result in Proposition \ref{prop:weak}. Hence, we proved that \textbf{Claim A} holds.
 
 By \textbf{Claim A}, we obtain that
 \begin{equation*}
		\Sc_g=0,~H_{\partial M}=\|\mr dF\|_{\tr}, \ \|\mr dF\|_{\tr_{g_S}}=n-1.
\end{equation*}
It immediately follows that $F\colon \partial M\to\sph^{n-1}$ is a local isometry. As $\sph^{n-1}$ is simply connected for  $n\geq 3$,  we obtain that $F$ is a global isometry. Hence, 
$(M^n,\partial M , g)$ is a smooth, compact manifold with nonempty boundary $(\partial M , g_{\partial M})$ isometric to the standard unit sphere $(\sph^{n-1},g_{\sph^{n-1}})$ and $\Sc_g \geq 0$. Hence, by \cite{Shi_Tam}*{Theorem 1} for $n=3$ and \cite{Shi_Tam_extension}*{Theorem 2} for $n \leq 7$ (see Appendix \ref{sec: Shi-Tam} for the precise statements), we obtain that $(M,\partial M , g)$ is isometric to the standard unit ball $(\mathbb{D}^n, g_{\mathbb{D}^n})$. The proof is finished.
\end{proof}

\subsection{Listing type scalar-mean comparison theorem}
In this subsection, we will prove Theorem \ref{thm: generalized_scalar_mean_comparsion}. Let us state Theorem \ref{thm: generalized_scalar_mean_comparsion} again below for reader's conveniences.

\begin{theorem} \label{mainthm: Listing_type}
     Suppose that $(M^n, \partial M, g), n =2, 3, 4$ is a smooth, compact Riemannian manifold with nonnegative scalar curvature $Sc_g \geq 0$ and mean convex boundary $H_{\partial M} > 0$. Let $F: (\partial M, g_{\partial M}) \rightarrow (\mathbb{S}^{n-1}, g_{\sph^{n-1}})$ be a smooth map with $\deg(F) \neq 0$. If $H_{\partial M} \geq \|\mr dF\|(n-1)$, then there exists constant $c > 0$ such that $F: (M^n,  cg) \rightarrow \displaystyle (\mathbb{D}^n, g_{\mathbb{D}^n})$ is an isometry.  
     \end{theorem}
\begin{proof}
    We still consider the case of the dimension $n=3,4$.

    \textbf{Claim B}:
    \begin{equation*}
        \Sc_g = 0\text{ on }M;\ H_{\partial M } = \|\mr dF\|_{\tr} = \|\mr dF\|(n-1)\text{ on }\partial M.
    \end{equation*}
    The argument of \textbf{Claim B} is similar to that of the \textbf{Claim A} in the proof of Lemma \ref{mainthm: geometric_verison} with minor changes. For example, line \eqref{eq:modification} is replaced by
 \begin{align}
&\ \ \lambda\int_M \varphi^2 \,\mr d\mathcal{H}_g^{n}\\
=&\int_M |\nabla \varphi|^2  +\frac{n-2}{4(n-1)}\Sc_g\varphi^2\,\mr d\mathcal{H}_g^{n} - \int_{\partial M} \varphi \frac{\partial \varphi}{\partial \nu} \,\mr d\mathcal{H}_g^{n-1} \nonumber\\
  =&\int_M |\nabla \varphi|^2\,\mr d\mathcal{H}_g^{n} +\frac{n-2}{4(n-1)}\Sc_g\varphi^2\,\mr d\mathcal{H}_g^{n} + \frac{n-2}{2(n-1)}\int_{\partial M } (H - \|\mr dF\|_{\tr}) \varphi^2 \,\mr d\mathcal{H}_g^{n-1}\nonumber\\
=&\int_M |\nabla \varphi|^2\,\mr d\mathcal{H}_g^{n} +\frac{n-2}{4(n-1)}\Sc_g\varphi^2\, \mr d\mathcal{H}_g^{n} \label{eq:modification2}\\
&+\frac{n-2}{2(n-1)}\int_{\partial M} \Big(  (H-\|\mr dF\|(n-1))+(\|\mr dF\|(n-1)-\|\mr dF\|_\tr)\Big)\varphi^2\,\mr d\mathcal{H}_g^{n-1}.\nonumber
\end{align}
We omit the rest of details for \textbf{Claim B}. 

As a result of \textbf{Claim B}, we obtain that for any $x\in \partial M$:
\begin{itemize}
    \item either $dF_x=0$ and $H(x)=0$,
    \item or $F$ at $x$ is a homothety, namely $g_{\partial M}=\|\mr dF\|^{-2}F^*g_{\sph^{n-1}}$, and $H=\|\mr dF\|(n-1)$.
\end{itemize}
By our assumption that $H>0$, the first kind of points does not exist. Therefore, we obtain that
\begin{equation} \label{eq: metric_scaling_relation}
  g_{\partial M } = \|\mr dF\|^{-{2}{}}F^*g_{\mathbb{S}^{n-1}}.  
\end{equation}
In particular, $F$ is a local diffeomorphism, hence a global diffeomorphism as $\sph^{n-1}$ is simply connected.

Moreover, if we set $h^{\frac{4}{n-2}} = \|\mr dF\|^{-2}$  on $\partial M$, then the equation in line (\ref{eq: metric_scaling_relation}) can be rewritten as 
$$g_{\partial M } = h^{\frac{4}{n-2}}F^*g_{\mathbb{S}^{n-1}} \text{ on } \partial M.$$

Consider the Dirichlet boundary problem as follows.
\begin{equation}\label{eq: harmonic}
    \begin{cases} 
		\Delta u=0,  \text{ in } M,\\
		u=h, \ \   \text{ on } {\partial M}.  
	\end{cases}
\end{equation}	
The standard elliptic theory and maximum principle shows that there exists a positive harmonic function $u$ that solves the Dirichlet boundary problem in line (\ref{eq: harmonic}). 

We further consider the conformal metric on $M$ given by
	$$g_u=u^{\frac{4}{n-2}}g, \text{ in } M.$$
\begin{itemize}
    \item The scalar curvature $\Sc_{g_u}$ of $g_u$  on $(M, \partial M, g_u)$ is given by
	$$\Sc_{g_u}=u^{-\frac{n+2}{n-2}}\big(-\frac{4(n-1)}{n-2}\Delta u+\Sc_g u\big)=0.$$

    \item The mean curvature $H_{g_u}$ of $\partial M$ with respect to $ g_u$ is given by 
    $$H_{g_u}=\frac{1}{u^{\frac{2}{n-2}}}\big(H_g+\frac{2(n-1)}{n-2}\frac{1}{u}\frac{\partial u}{\partial\nu}\big)=(n-1)+\frac{2(n-1)}{n-2}\frac{1}{u^{\frac{n}{n-2}}}\frac{\partial u}{\partial \nu},$$
    where $\nu$ is the unit, outer normal vector field of $\partial M$.

    \item Under the map $F$, $(\partial M, g_u)$ is isometric to $(\sph^{n-1}, g_{\sph^{n-1}})$.
\end{itemize}

Finally, we calculate the integral of $H_{g_u}$ on $(\sph^{n-1}, g_{\sph^{n-1}})$:
\begin{align}
		&\ \ \ \int_{\partial M} H_{g_u}\,\mr d\mathcal{H}_{g_{\sph^{n-1}}}^{n}\\
        &=\int_{\partial M} (n-1)\,\mr d\mathcal{H}_{g_{\sph^{n-1}}}^{n-1}+\frac{2(n-1)}{n-2}\int_{\partial M} \frac{1}{u^{\frac{n}{n-2}}}\frac{\partial u}{\partial \nu_S}\cdot u^{\frac{2(n-1)}{n-2}} \,\mr d\mathcal{H}_{g}^{n-1}\nonumber \\
        &=\int_{\partial M} (n-1)\,\mr d\mathcal{H}_{g_{\sph^{n-1}}}^{n-1}+\frac{2(n-1)}{n-2}\int_{\partial M} u\cdot \frac{\partial u}{\partial\nu_S}\,\mr d\mathcal{H}_{g}^{n-1}\nonumber \\
		&=\int_{\partial M} (n-1)\,\mr d\mathcal{H}_{g_{\sph^{n-1}}}^{n-1}+ \frac{2(n-1)}{n-2}\int_{M} u\Delta u \,\mr d\mathcal{H}_g^n + \frac{2(n-1)}{n-2}\int_{\partial M } |\nabla u|^2 \,\mr d\mathcal{H}_g^n \label{eq:integral}\\
  		&=\int_{\partial M} (n-1)\,\mr d\mathcal{H}_{g_{\sph^{n-1}}}^{n-1} + \frac{2(n-1)}{n-2}\int_{M } |\nabla u|^2\,\mr d\mathcal{H}_g^n\nonumber  \\
		&\geq\int_{\partial M} (n-1)\,\mr d\mathcal{H}_{g_{\sph^{n-1}}}^{n-1}.\nonumber 
\end{align}

To summarize, we proved that $(M, \partial M, g_u)$ is  a smooth, compact Riemannian manifold such that
\begin{enumerate}
    \item $\Sc_{g_u} =0$ on $M$,
    \item $(\partial M, g_u)$ is isometric to $(\sph^{n-1}, g_{\sph^{n-1}})$,
    \item $\displaystyle\int_{\partial M} H_{g_u}\,\mr d\mathcal{H}_{g_{\sph^{n-1}}}^{n} \geq \int_{\partial M} (n-1)\,\mr d\mathcal{H}_{g_{\sph^{n-1}}}^{n-1}$.
\end{enumerate}
By \cite{Shi_Tam}*{Theorem 1} for $n=3$ and \cite{Shi_Tam_extension}*{Theorem 2} for $n\leq 7$, we obtain that $(M^n, \partial M, g_u)$ is isometric to $(\mathbb{D}^n, \mathbb{S}^{n-1}, g_{\mathbb{D}^n})$, and
$$\int_{\partial M} H_{g_u}\,\mr d\mathcal{H}_{g_{\sph^{n-1}}}^{n} =\int_{\partial M} (n-1)\,\mr d\mathcal{H}_{g_{\sph^{n-1}}}^{n-1}.$$
As a result, the (last) inequality of line \eqref{eq:integral} is an equality. This implies that  $\nabla u=0 $ in $M$. Hence, $u$ is positive constant in $M$ and then $h$ is a positive constant function on $\partial M$. We finished the proof.
\end{proof}

\section{Lipschitz scalar-mean rigidity} \label{sec: Lip_scalar_mean_rigdity}
In this section, we prove Theorem \ref{thm:Lipschitz} stated as follows.
\begin{theorem}\label{thm:Lipschitz2}
      Suppose that $(M^n, \partial M, g), n=2, 3, 4$ is a smooth, compact Riemannian manifold with nonnegative scalar curvature $\Sc_g \geq 0$ and uniformly positive mean curvature  $H_{\partial M} \geq n-1$. If $F: (\partial M, g_{\partial M}) \to (\mathbb{S}^{n-1}, g_{\mathbb{S}^{n-1}})$ is a distance non-increasing \textbf{Lipschitz} map of $\deg(F)  \neq 0$, then $F$ is a \textbf{smooth} isometry, and $(M, g)$ is isometric to $(\mathbb{D}^n, g_{\mathbb{D}^n})$.
\end{theorem}

We first introduce an oriented trace function for oriented vector spaces. Recall that an oriented vector space is a vector space together with a given choice of orientation.  
\begin{definition}\label{def:orientedTrace}
Let $U,V$ be $n$-dimensional oriented vector spaces with inner products $g,g'$, and $T\colon U\to V$ a linear transformation. The oriented trace function of $T$ is defined by
    $$[T]_\tr\coloneqq \sup_{\{u_i\},\{v_i\}}\sum_{i=1}^n\langle Tu_i,v_i\rangle_{g'},$$
    where the supremum is taken among all oriented orthonormal basis $\{u_i\}_{1\leq i\leq n}$ of $(U,g)$ and oriented orthonormal basis $\{v_i\}_{1\leq i\leq n}$ of $(V,g')$.

\end{definition}

We shall possibly write $[T]_{\tr,g}=[T]_\tr$ whenever it is necessary to emphasize its dependency on the inner product $g$. The oriented trace function has the properties as follows.

\begin{lemma}\label{lemma:orientedTrace}
    If $U,V$ are $n$-dimensional oriented vector spaces ($n\geq 2$) with inner products $g,g'$ respectively, then the oriented trace function is sublinear and nonnegative. Moreover, if  $T\colon U\to V$ is a linear transformation, then $$[T]_\tr\leq \|T\|_\tr.$$
    In particular, the equality holds if and only if
    \begin{itemize}
        \item either $T$ is not invertible,
        \item or $T$ is invertible and $T$ is orientation preserving.
    \end{itemize}
\end{lemma}
\begin{proof}
By the definition of oriented trace function, it is direct that 
$$[sT]_\tr=s[T]_\tr,~\forall s\geq 0,\text{ for any } T\colon U\to V$$
and
$$[T_1+T_2]_\tr\leq [T_1]_\tr+[T_2]_\tr,\text{ for any }  T_1,T_2\colon U\to V.$$

    Given any  $T\colon U\to V$ linear transformation, we consider the singular value decomposition of $T$, namely the orthonormal basis $\{e_i\}_{1\leq i\leq n}$ of $U$ and $\{f_i\}_{1\leq i\leq n}$ of $V$ with
    \begin{equation} \label{eq: svd}
        Te_i=\lambda_if_i
    \end{equation}
    for some $\lambda_i\geq 0$. We may assume that $\{e_i\}_{1\leq i\leq n}$ is an oriented, orthonormal basis of $U$,  and note that one of the basis $\{\pm f_n,f_1,f_2,\ldots,f_{n-1}\}$ forms an oriented orthonormal basis of $V$. A direct check shows that
    $$\langle Te_1,\pm f_n\rangle+\sum_{i=2}^n\langle Te_i,f_{i+1}\rangle=0.$$
    Hence 
    $$[T]_\tr\geq 0.$$
    \vspace{2mm}
    
    Note that the definitions of trace norm and trace function indicates direclty that $$[T]_\tr\leq \|T\|_\tr.$$ 
    Moreover, if $T$ is not invertible, without loss of generality,  we may assume that $\lambda_1=0$. Note that one of the basis $\{\pm f_1,f_2,\ldots,f_n\}$ forms an oriented, orthonormal basis of $V$, we have
    $$[T]_\tr\geq \sum_{i=2}^n\langle Te_i,f_i\rangle=\|T\|_\tr.$$
    Hence, we obtain 
    $$[T]_\tr=\|T\|_\tr.$$

    Next, if $T$ is invertible and $[T]_\tr=\|T\|_\tr$, then we suppose that,  for the oriented orthonormal basis $\{u_i\}_{1\leq i\leq n}$ of $U$ and oriented orthonormal basis $\{v_i\}_{1\leq i\leq n}$ of $V$, we have
    $$[T]_\tr=\sum_{i=1}^n\langle Tu_i,v_i\rangle.$$
    Hence, we obtain
    \begin{equation}\label{eq:orTrace}
        [T]_\tr=\sum_{i=1}^n\langle Tu_i,v_i\rangle= \sum_{i=1}^n|\langle Tu_i,v_i\rangle|= \|T\|_\tr.
    \end{equation}
    
    \vspace{2mm}
    
    Finally, given the singular value decomposition of $T$ in line \ref{eq: svd}, we assume that
    $$u_i=\sum_{j=1}^n a_i^je_j,~v_i=\sum_{k=1}^n b_i^k f_k.$$
    Here, we denote $A=(a_i^j)_{n\times n}$ and $B=(b_i^j)_{n\times n}$.  Note that  $\{e_i\}_{1\leq i\leq n}$ is oriented by our assumption, we have $\det(A)>0$. The equality in line \eqref{eq:orTrace} yields that
$$\sum_{j=1}^n\lambda_j\Big(\sum_{i=1}^na_i^jb_i^j\Big)=\sum_{i=1}^n\Big|\sum_{j=1}^n\lambda_ja_i^jb_i^j\Big|=\sum_{j=1}^n\lambda_j.$$
Since $T$ is invertible,  we have $\lambda_j>0$ for each $j$. Therefore, for each $j$, the Cauchy--Schwarz inequality
$$\sum_{i=1}^na_i^jb_i^j\leq \sqrt{\sum_{i=1}^n(a_i^j)^2\sum_{i=1}^n(b_i^j)^2}=1$$
is indeed an equality. Therefore, $AB^T$ is a matrix whose diagonal entries are all $1$. Since $AB^T$ is also orthogonal, we obtain that $AB^T=I$, namely $A=B$. As
$$f_i=\sum_{k=1}^n b_k^i v_k,$$
and $\det(B^T)=\det(B)=\det(A)>0$, the basis $\{f_i\}_{1\leq i\leq n}$ is also oriented. Therefore, $T$ is orientation preserving.
\end{proof}

The proof of Theorem \ref{thm:Lipschitz} is indeed similar to that of Theorem \ref{mainthm: scalar_mean_curvature_comparsion}. We only sketch the proof here. We first need an extremality theorem for mean curvature with $[~\cdot~]_\tr$ lower bound.

\begin{proposition}\label{prop:weakLipschitz}
Suppose that $(M^n, \partial M, g)$ is a smooth,  compact Riemannian manifold with nonempty boundary $\partial M$ and nonnegative scalar curvature $\Sc_g \geq 0$ in $M$. If $ F \colon (\partial M,g|_{\partial M})\to(\sph^{n-1},g_{\sph^{n-1}})$ is a smooth map such that \begin{equation}\label{eq:delta_f}
		H_g\geq [\mr d F]_{\tr}+\delta \ \text{ on } \partial M
	\end{equation}
for some fixed positive constant $\delta > 0$ and $n=2, 3, 4$, then $\deg(F)=0$.
\end{proposition}
\begin{proof}
We always assume that $M$ is oriented and $\deg(F)>0$. Otherwise, we consider the double cover of $M$.

When $n=2$, the proposition also follows from the Gauss--Bonnet formula. On $M$, we have
\begin{equation*}
   \frac{1}{2} \int_M \Sc_g\,\mr d\mathcal{H}_g^2 + \int_{\partial M } k_{g}\,\mr d\mathcal{H}_g^1 =  2\pi\chi(M) \leq 2\pi,
\end{equation*}
where the geodesic curvature $k_{g}$ is equal to the mean curvature $H_{g}$. By definition, 
$$[F]_\tr=\frac{\mr d(F^*\theta)}{\mr ds},$$
where $\theta$ and $s$ are the arc length parameters of $\sph^1$ and $\partial M$, respectively. By our assumption and $\deg(F) >0$, we obtain that
\[\int_M \Sc_g\,\mr d\mathcal{H}_g^2 + 2\int_{\partial M } k_{\partial M } \,\mr d\mathcal{H}_g^1 \geq {4\pi\cdot \deg(F) + \delta \cdot \mathcal{H}_g^1(\partial M )}.\]
Hence, we reach that $\mathcal{H}_g^1(\partial M ) = 0$, which is a contradiction.

The general case is proved by induction. Assume the conclusion holds for $n-1$. We shall use the same process as in the proof of Proposition \ref{prop:weak} and obtain the smooth submanifold
$$(Y^{n-1},Z^{n-2}=\partial Y,g_f),$$
of $(M^n,S^{n-1}=\partial M,g),$ where $g_f=f^{\frac{4}{m-2}} g$ and $f$ is given in line \eqref{eq:Lboundary}. We have $\Sc_{g_f}\geq 0$ as in line \eqref{eq:Sc>=0}, and the mean curvature given by
\begin{align*}
    H_{Z, g_{f}} =  f^{-\frac{2}{m-2}}\left(H_{Z, g}  + \frac{m-1}{2(m-2)}\frac{1}{ f}\frac{\partial f}{\partial \nu_Z}\right)
    ={f^{-\frac{2}{m-2}}} \left(\frac{H_S}{\sin(J)}  - \frac{1}{\sin(J)} \frac{\partial J}{\partial\n}\right)
\end{align*}
as in line \eqref{eq:H>}, where $\n$ is the upper unit normal vector of $Z$ in $S$.

We define $F_{n-2}= P_{n-2} \circ F_{n-1}$, where $P_{n-2}$ is the projection from $\sph^{n-1}$ to the equator. Let $\nabla J$ be the gradient of $J$, which is the unit vector field on $\sph^{n-1}$ along the geodesics from the south pole to the north pole. For any point $z\in Z$, let $\{u_i\}_{1\leq i\leq n-2}$ be an oriented orthonormal basis of $T_z Z$ with respect to $g_f$, and $\{v_i\}_{1\leq i\leq n-2}$ an oriented orthonormal basis of $T_{F_{n-2}(z)}\sph^{n-2}$. Then
$$\Big\{\normal,f^{\frac{2}{m-2}}u_1,\cdots,f^{\frac{2}{m-2}}u_{n-2}\Big\}$$
is an oriented orthonormal basis of $T_zS$, and
$$\Big\{\nabla J,\frac{1}{\sin(J)}(\mr dP_{n-2})^{-1}v_1,\cdots,\frac{1}{\sin(J)}(\mr dP_{n-2})^{-1}v_{n-2}\Big\}$$
is an oriented orthonormal basis of $T_{F_{n-1}(z)}\sph^{n-1}$. 

Therefore, by Definition \ref{def:orientedTrace}, we have
\begin{align*}
    H_S-\delta\geq [\mr dF_{n-1}]_{\tr,g}=&\langle \mr dF_{n-1}(\n),\nabla J\rangle+\sum_{i=1}^{n-2}\langle \mr dF_{n-1} (f^{\frac{2}{m-2}}u_i),\frac{1}{\sin(J)}(\mr dP_{n-2})^{-1} (v_i)\rangle\\
    =&\frac{\partial F_{n-1}^*J}{\partial \n}+\sum_{i=1}^{n-2}\sin(J)f^{\frac{2}{m-2}}\langle (\mr dP_{n-2}\circ \mr dF_{n-1}) (u_i), v_i\rangle
\end{align*}
Since $\{u_i\}$ and $\{v_i\}$ are arbitrary, we obtain that
$$H_S-\delta\geq \frac{\partial J}{\partial \n}+\sin(J)f^{\frac{2}{m-2}}[\mr dF_{n-2}]_{\tr,g_f}.$$
Therefore, we have
$$H_{Z,g_f}\geq [\mr dF_{n-2}]_{\tr,g_f}+\delta\cdot f^{-\frac{2}{m-2}}\frac{1}{\sin(J)}.$$
Since $F_{n-1}(\partial Y)$ stays away from the poles and $f$ is strictly positive on $Y$, we get that 
$$\widetilde\delta=\delta\cdot \inf_{Z}\frac{1}{\sin(J)}f^{-\frac{2}{m-2}}>0.$$
Consequently, we obtain a smooth compact Riemannian manifold $(Y^{n-1}, \partial Y, g_f)$  of dimension $(n-1)$ with

\begin{enumerate}

    \item Nonnegative scalar curvature:  $$ \Sc_{g_{f}} \geq 0 \text{ in }  Y.$$
    
    \item \label{item: mean_curvature_lower_bound_2} Mean curvature lower bound: there exists a smooth map $$F_{n-2}: (\partial Y, g_{f}|_{\partial M}) \rightarrow (\sph^{n-2}, g_{\sph{n-2}})$$ such that
    \[H_{\partial Y, g_f} \geq [\mr dF_{n-2}]_{\tr_{g_f}} + \Tilde{\delta}\]
    for some positive constant $\Tilde{\delta} > 0$ and $\deg(F_{n-2}) =\deg(F_{n-1})$.
\end{enumerate}
This finishes the proof by the induction hypothesis.
\end{proof}

\vspace{2mm}
Now we are ready to prove Theorem \ref{thm:Lipschitz2}.
\begin{proof}[Proof of Theorem \ref{thm:Lipschitz2}]
    The statement clearly holds for $n=2$. We consider $n\geq 3$.

    \textbf{Claim C:} Under the assumption of Theorem \ref{thm:Lipschitz}, we have
 \begin{equation}\label{eq:equalLipschitz}
     \Sc_g=0\text{ on }M; H_{g}=[\mr dF]_\tr=\|\mr dF\|_\tr=n-1\text{ on }\partial M.
 \end{equation}   

 Let us argue by contradiction. Suppose that at least one of these equalities fails at some point in $M$. Similar as the proof of Theorem \ref{mainthm: scalar_mean_curvature_comparsion}, the lowest eigenvalue $\lambda$ of the Neumann boundary problem is positive:
 	\begin{equation} 
  \left\{
		\begin{aligned}
			& -\Delta \varphi +\frac{n-1}{4(n-2)} \Sc_g\varphi = \lambda \varphi,\\
			&\frac{\partial \varphi}{\partial \nu}=-\frac{n-2}{2(n-1)}(H-[\mr dF]_{\tr})\varphi.
		\end{aligned}\right.
	\end{equation}
 Here $[\mr dF]_{\tr}$ is only an $L^\infty$-function on $\partial M$. Therefore, there exists a smooth map $F'\colon \partial M\to\sph^n$ with
 $$\begin{cases}
     \displaystyle\sup_{x\in\partial M}d(F(x),F'(x))<\varepsilon,\\
     \|\mr dF-\mr dF'\|_{L^p(\partial M)}<\varepsilon,
 \end{cases}$$
 for some small $\varepsilon>0$ and large $p$, such that the lowest eigenvalue $\lambda'$ of the Neumann boundary problem is positive:
 	\begin{equation} 
  \left\{
		\begin{aligned}
			& -\Delta \varphi +\frac{n-1}{4(n-2)} \Sc_g\varphi = \lambda' \varphi,\\
			&\frac{\partial \varphi}{\partial \nu}=-\frac{n-2}{2(n-1)}(H-[\mr dF']_{\tr})\varphi.
		\end{aligned}\right.
	\end{equation}

 Therefore, as computed in the proof of Theorem \ref{mainthm: scalar_mean_curvature_comparsion}, we obtain a new metric on $M$ that satisfies the conditions in Proposition \ref{prop:weakLipschitz}. This leads to a contradiction and proves \textbf{Claim C}.

 Therefore, all the equality in line \eqref{eq:equalLipschitz} holds. In particular, by Lemma \ref{lemma:orientedTrace}, $\mr dF$ is almost everywhere an orientation preserving isometry. By \cite{cecchini2022lipschitz}*{Theorem 2.4} and the Myers--Steenrod Theorem \cite{Myers_Steenrod}, $F$ is a smooth isometry. It follows that $(M,g)$ is a Euclidean flat disk.
\end{proof}

\appendix

\section{Capillary mu-bubble and its variation} \label{sec: capillarysurface}

In this section, we will first set up the capillary $\mu$-bubble problem in a general context, and then we will present the basic calculations for the first and second variations of the capillary $\mu$-bubble. Our primary focus is to deal with the boundary quantities since the other calculations are quite standard in the standard textbook. This section is a refined version of the calculations from \cites{Gromov_four_lectures,Ambrozio_Rigidity, Li_polyhedron_three}, see \cites{ZZ-cmc,ZZ-pmc, chodosh2024_improvedregularity} for the further studies of the capillary $\mu$-bubble.


Suppose that $(M^n, \partial M, g)$ is a complete Riemannian manifold with nonempty boundary $S = \partial M$. Let $\Omega$ be a domain with boundaries, we write $ (Y^{n-1}, \partial Y) =\partial \Omega \cap \mathring{M} $,
$Z = \partial Y \subset S =\partial M$ and $ \nu_Y$ be the upward (outer) unit normal vector field of $Y$ in $M$. Now we define
\begin{equation}
    \mu_c = \mu(x)\,\mr d\mathcal{H}_g^n(x) + \mu_{\partial}(x)\,\mr  d\mathcal{H}_g^{n-1}(x).
\end{equation}

Moreover, we define the capillary $\mu$-bubble functional as follows.

\begin{definition} We introduce the capillary $\mu$-bubble as follows.
    $$\mathcal{A}_c(\Omega) = \mathcal{H}_g^{n-1}(Y) - \left( \int_{\Omega} \mu(x) \,\mr d\mathcal{H}_g^n(x) + \int_{\partial \Omega^* \cap S}\mu_{\partial}(x)\,\mr d\mathcal{H}_g^{n-1}(x) \right).$$
for any $\Omega$ in $\mathcal{C}$. Here 
\[\mathcal{C} = \left\{\text{Caccioppoli sets } \Omega \subset X \text{ {with certain given topological properties}} \right\}.\]
\begin{itemize}
    \item A domain $\Omega \subset M$ is said to be $\mc A_c$ stationary if it is a critical point of $\mathcal{A}_c$ {among the class $\mathcal{C}$}.
    \item A domain $\Omega \subset M$ is said to be a stable $\mu$-bubble if $\Omega$ is a minimizer of $\mathcal{A}_c$ among the class $\mathcal{C}$. 
\end{itemize}
\end{definition}

\vspace{2mm}
Our next goal is to calculate the variation of the capillary $\mu$-bubble and then study the curvature relations on the boundary.

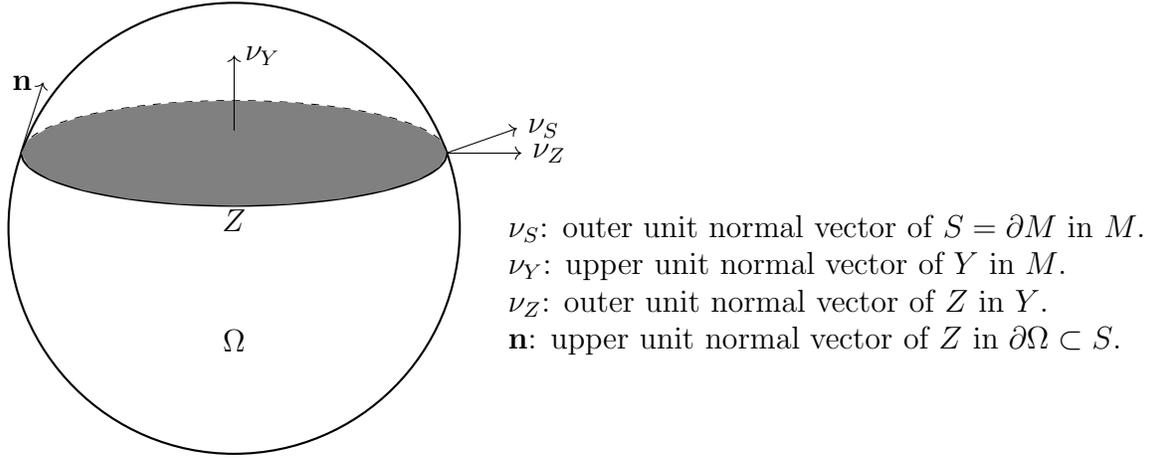
\begin{figure}[h]\label{fig:mu-bubble}
	\begin{tikzpicture}
		\draw[thick] (3,0) arc(0:360:3);
		\draw[dashed,domain=0:180,smooth] plot({2.83*cos(\x)},{0.7*sin(\x)+1});
		\draw[thick,domain=180:360,smooth] plot({2.83*cos(\x)},{0.7*sin(\x)+1});
		\node at (-1,1)[fill=white]{$Y = \partial \Omega \cap \mathring{M}$};
		\fill[gray,opacity=0.2] plot[domain=0:360] ({2.83*cos(\x)},{0.7*sin(\x)+1});
		\node at (0,0.1){$Z$};
		\draw[->] (0,1.3) -- (0,2.3) node[right] {$\nu_Y$};
		\draw[->] (2.82,1) -- (3.76,1.33) node[right] {$\nu_S$};
		\draw[->] (2.82,1) -- (3.82,1) node[right] {$\nu_Z$};
		\draw[->] (-2.82-0.02,1) -- (-2.49-0.05,1.94) node[left] {$\textbf{n}$};
		\draw (4-.5,2-2) node[right] {$\nu_S$: outer unit normal vector of $S=\partial M$ in $M$.};
		\draw (4-.5,1.5-2) node[right] {$\nu_Y$: upper unit normal vector of $Y$ in $M$.};
		\draw (4-.5,1-2) node[right] {$\nu_Z$: outer unit normal vector of $Z$ in $Y$.};
		\draw (4-.5,0.5-2) node[right] {$\textbf{n}$: upper unit normal vector of $Z$ in $\partial\Omega\subset S$.};
        \node at (0,-1.5){$\Omega$};
	\end{tikzpicture}
 \caption{Capillary $\mu$-bubble setup} \label{setup}
\end{figure}

Note that the variation of the domain $\Omega \in \mathcal{C}$ is equivalent to the variation of its boundary $Y = \partial^*\Omega $.  Hence, we mainly focus on boundary $(Y, \partial Y)$. Suppose that $(Y, \partial Y)$ is a smooth hypersurface in $M$  and $(Y_t, \partial Y_t)$ is a family of hypersurfaces in $M$ such that $\partial Y_t \subset S = \partial M$ and $(Y_0, \partial Y_0) = (Y, \partial Y)$ for $t \in (-\varepsilon, \varepsilon), \varepsilon > 0$. Here, we denote by 

\begin{itemize}
    \item $\nu_{Y_t}$ the unit, upper normal vector field of $Y_t$ in $M$,

    \item $\nu_{Z_t}$  the unit, outer normal vector field of $Z_t$ in $Y_t$,

    \item $\nu_S$ the unit, outer normal vector field $S$ in $M$,

    \item $\n_t$ the unit, upper normal vector field of $Z_t$ in $S$.
\end{itemize}

Moreover, we define $J_t(z)$ by the contact angle between $Y_t$ and $S$ at the intersection point $z \in Z_t$ = $\partial Y_t$, then
\begin{equation} \label{eq: angle_formula}
    \cos(J_t(z)) = -\nu_{Y_t}(z) \cdot \nu_{S}(z) = \nu_{Z_t} \cdot \n_t.
\end{equation}
Note that $\nu_{Z_t}, \nu_{S}, \nu_{Y_t}$ can be viewed as the unit, normal vector fields of $Z_t$ in $M$ and then they are in the same plane. Hence, for any $z \in Z$, we obtain
\begin{equation} \label{eq: angle_relations}
    \nu_{S}(z) = -\cos(J_t(z)) \cdot \nu_{Y_t}(z) + \sin(J_t(z))\cdot \nu_{Z_t}(z).
\end{equation}

Next we consider the admissible deformation of $Y$: $f: Y \times (-\epsilon, \epsilon) \rightarrow M$ such that $f_t: Y \rightarrow M$ defined by $f_t(y) = f(y,t)$ is an embedding in $M$ with 
$$f_t(\mathring{Y}) \subset \mathring{M}, \ f_t(\partial Y) \subset  S,\ f_0(y) = y \text{ for any } y \in Y.$$

Now we define the variational vector field $\partial_t(y) =: \frac{\partial f}{\partial t}(y, t), t \in (-\epsilon, \epsilon)$. Note that $Y|_{Z_t} \in TS$ and denote $$\varphi(y,t) = g(\partial_t,\nu_{Y_t}) \text{ for any }y \in {Y}  .$$
Moreover, on the boundary $z \in Z_t$, we obtain that
\begin{equation}
    \partial_t (z)=  \partial^Z_t(z) + \frac{\varphi(z,t)}{\sin(J(z,t))}\cdot \n(z,t). 
\end{equation}
Here,  $\partial^Z_t(z)$ is the tangential part of $\partial_t(z)$ onto $Z_t$ and $\n$ is the unit upward normal vector field of $Z$ in $S$.

Hence, we reach that
\begin{lemma} With the notation above, we obtain
   \[ \mathcal{A}^\prime_c(t) = \int_{Y_t} \left(H_{Y_t} - \mu \right) \cdot  \varphi \,\mr d\mathcal{H}_g^{n-1}  + \int_{Z_t} \left(\frac{{\cos(J_t) - \mu_{\partial}}}{\sin(J_t)}\right)\cdot  \varphi \,\mr d\mathcal{H}^{n-2}_g.  \]
   Here, $H_{Y_t}$ is the mean curvature of $Y_t$ with respect to $\nu_{Y_t}$ and $J_t$ is the contact angle $Y_t$ and $S$ at the intersection points. As a result, $Y$ is a stationary hypersurface of $\mathcal{A}_c$ if and only if
   \begin{equation}
   H_Y(y) = \mu(y) \text{ in } Y; \ 
       \cos(J(z)) = \mu_{\partial}(z) \text{ on } Z
   \end{equation}

   \begin{proof}

       By a basic calculation(see \cite{Ambrozio_Rigidity}*{Appendix}), we obtain that
       \begin{align*}
          \frac{d}{dt}\mathcal{H}_g^{n-1}(Y_t) &= \int_Y H_{Y_t} \cdot \varphi \,\mr d\mathcal{H}_g^{n-1} + \int_{Z_t}g(\nu_{Z_t}, \partial_t)\,\mr  d\mathcal{H}_g^{n-2}\\
          &=  \int_Y H_{Y_t}\cdot \varphi \,\mr  d\mathcal{H}_g^{n-1} + \int_{Z_t}g(\nu_{Z_t}, \partial_t^{Z_t} + \frac{\varphi}{\sin(J_t)}\n_t) \,\mr d\mathcal{H}_g^{n-2}\\
          &=  \int_Y H_{Y_t}\cdot \varphi \,\mr d\mathcal{H}_g^{n-1} + \int_{Z_t}g(\nu_{Z_t},  \frac{\varphi}{\sin(J_t)}\n_t) \,\mr  d\mathcal{H}_g^{n-2}\\
           & =  \int_{Y_t} H_{Y_t} \cdot \varphi \,\mr d\mathcal{H}_g^{n-1} + \int_{Z_t} \frac{\cos(J_t)}{\sin(J_t)} \cdot \varphi \,\mr d\mathcal{H}_g^{n-2}
       \end{align*}
 Moreover, a direct calculation implies that
       \begin{equation*}
           \frac{d}{dt}\int_{\Omega_t} \mu \,\mr d\mathcal{H}_g^{n-1} = \int_{Y_t} \mu \cdot \varphi\,\mr d\mathcal{H}_g^{n-1}.
       \end{equation*}
        \begin{equation*}
           \frac{d}{dt}\int_{\partial \Omega_t \cap S} \mu_{\partial}(z) \,\mr d\mathcal{H}_g^{n-1} = \int_{Z_t} \frac{\mu_\partial}{\sin(J_t)}\cdot \varphi \,\mr  d\mathcal{H}_g^{n-1}.
       \end{equation*}
    Hence, we obtain
     \[ \mathcal{A}^\prime_c(t) = \int_{Y} \left(H_{Y_t} - \mu\right) \cdot  \varphi \,\mr d\mathcal{H}_g^{n-1}  + \int_{Z} \frac{\cos(J_t) - \mu_{\partial}}{\sin(J_t)} \cdot \varphi\,\mr d\mathcal{H}^{n-2}_g. \]
     Therefore, $\Omega$ is a stationary capillary $\mu$-bubble of $\mathcal{A}_c$ if and only if
   \begin{equation*}
   H_Y(y) = \mu(y) \text{ in } Y; \ 
       \cos(J(z)) = \mu_{\partial}(z) \text{ on } Z.
   \end{equation*}
   \end{proof}
\end{lemma}

\begin{lemma}\label{lemma: second_variation_formula}
   With the notations above, if $\Omega$ is  a stationary capillary $\mu$-bubble of $\mathcal{A}_c$, then
        \begin{align*}
       \mathcal{A}^{\prime\prime}(0) =&\int_Y |\nabla \varphi|^2  - \left(\Ric_g(\nu_Y, \nu_Y) + \|A_Y\|^2 + \partial_{\nu_Y} \mu \right) \cdot \varphi^2 \,\mr d\mathcal{H}_g^{n-1}  \\
       + & \int_Z \left(  H_Z - \frac{H_S}{\sin(J)} -  \cot(J)H_Y - \frac{1}{\sin^2(J)}\frac{\partial \mu_\partial}{\partial \n} \right)\cdot \varphi^2   + 2(\nabla_{\partial_t^Z}J) \cdot \varphi \,\mr d\mathcal{H}_g^{n-2}.
   \end{align*}
Here, $H_Z$ is the mean curvature $Z$ in $Y$ with respect to $\nu_Z$, $H_S$ is the mean curvature of $S$ in $M$ with respect to $\nu_S$, and $H_Y$ is the mean curvature of $Y$ in $M$ with respect to $\nu_Y$. In particular, if  $\partial_t^Z = 0$, we obtain,
\begin{align} \label{eq: second_variation_proper}
       \mathcal{A}^{\prime\prime}(0) =&\int_Y |\nabla \varphi|^2  - \left(\Ric_g(\nu_Y, \nu_Y) + \|A_Y\|^2 + \partial_{\nu_Y} \mu \right) \cdot \varphi^2 \,\mr d\mathcal{H}_g^{n-1}  \\
       + & \int_Z \left(  H_Z - \frac{H_S}{\sin(J)} -  \cot(J)H_Y - \frac{1}{\sin^2(J)}\frac{\partial \mu_\partial}{\partial \n} \right)\cdot \varphi^2  \,\mr d\mathcal{H}_g^{n-2}.
   \end{align}
   \begin{proof}
  
By the classical variational formula(see \cite{Ambrozio_Rigidity}*{Appendix}), we obtain that
\begin{equation*}
    \frac{\partial H_{Y_t}}{\partial t}= -\Delta_{Y_t} \varphi - \left(\|A\|^2 + \Ric_g(\nu_{Y_t}, \nu_{Y_t})\right)\varphi\ + \nabla^g_{\partial^{Y_t}_t} H_t.
\end{equation*}
Here, $\nabla^g$ is the Levi-Civita connection induced by the Riemannian metric $g$ on $M$.

Let us work on $Z_t$ and then view  $\nu_{Z_t}, \nu_{S}, \nu_{Y_t}$ as the unit normal vector field of $Z_t$ in $X$. Note that the angle decomposition in (\ref{eq: angle_relations})
\begin{equation*} 
    \nu_{S}(z) = -\cos(J_t(z)) \cdot \nu_{Y_t}(z) + \sin(J_t(z))\cdot \nu_{Z_t}(z),
\end{equation*} we obtain
       \begin{equation}\label{eq: second_fundamental_form_relations}
           \tr_{g_{Z_t}}(A_{\nu_{S}}) = -\cos(J_t(z))\cdot \tr_{g_{Z_t}}(A_{\nu_{Y_t}}) + \sin(J_t(z))\cdot \tr_{g_{Z_t}}(A_{\nu_{Z_t}}).
       \end{equation}
       Here $\tr_{g_{Z_t}}( \cdot )$ stands for taking the trace on $Z_t$ with respect to the metric $g_{Z_{t}}$ and $A_\nu$ stands for the second fundamental from $Z$ with respect to $\nu$ in $M$ for any unit normal vector field $\nu$ of $Z$. Then, by taking trace, line (\ref{eq: second_fundamental_form_relations}) implies that
       \begin{equation} \label{eq: mean_curvature}
          \sin(J_t(z))\cdot  H_{Z_t} = 
           \tr_{g_{Z_t}}(A_{\nu_{S}}) +\cos(J_t(z))\cdot \tr_{g_{Z_t}}(A_{\nu_{Y_t}}).
       \end{equation}
       
       Moreover, let us further work on $Y_t$ (resp. $S$) in $M$ (resp. $M$), by the definition of second fundamental form, we reach,
       \begin{itemize}
           \item  Let us consider the second term on the right in line (\ref{eq: mean_curvature})
           \begin{align*}
           H_{Y_t} = \tr_{g_{Y_t}}(A_{\nu_{Y_t}})&=  \tr_{g_{Z_t}}(A_{\nu_{Y_t}}) + g(\nabla_{\nu_{Z_t}}\nu_{Y_t},\nu_{Z_t}).
       \end{align*}
Hence,
\begin{align*}
    \cos(J_t(z))\cdot \tr_{g_{Z_t}}(A_{\nu_{Y_t}}) &= \cos(J_t(z))H_{Y_t}-\cos(J_t(z))g(\nabla_{\nu_{Z_t}}\nu_{Y_t},\nu_{Z_t})\\
    &=  \cos(J_t(z))H_{Y_t}-\cos(J_t(z))A_{\nu_{Y_t}}(\nu_{Z_t}, \nu_{Z_t})
\end{align*}

\item Let us consider the first term on the left in line (\ref{eq: mean_curvature})
\begin{align*}
    \tr_{g_{Z_t}}(A_{\nu_{S}}) &=\tr_{g_{S}}(A_{\nu_{S}})  - g(\nabla_{\n_t} \nu_S, \n_t)\\
    &= H_S  - g(\nabla_{\n_t} \nu_S, \n_t)\\
    &= H_S - A_{\nu_S}(\n_t, \n_t).
\end{align*}

\end{itemize}
Hence, the calculations above imply that
\begin{align*}\label{eq: mean_curvature_relations}
    \sin(J_t)\cdot H_{Z_t} - H_S - \cos(J_t)\cdot H_{Y_t}  =-A_{\nu_S}(\n, \n) -\cos(J_t) \cdot A_{\nu_{Y_t}}({\nu_{Z}},\nu_{Z})
\end{align*}

\vspace{2mm}
Next, let us calculate $\frac{d}{dt}\cos(J_t)|_{t = 0}$ as follows. By the angle expression (\ref{eq: angle_formula}) and (\ref{eq: angle_relations}), we obtain
           \begin{align*}
          \frac{d}{dt} \cos(J_t(z))&=-{\partial_t} (g( \nu_{Y_t} , \nu_{S}))\\
          &=- g(\nabla_{\partial_t} \nu_{Y_t},  \nu_S)  -g(\nu_{Y_t}, \nabla_{\partial_t} \nu_S)\\
          &= - g(\nabla_{\partial_t^Y} \nu_{Y_t},  \nu_S) + g(\nabla^Y{ \varphi},{ \nu_S}) -g(\nu_{Y_t},\nabla_{\partial_t} \nu_S).
       \end{align*}
Here, $\partial_t^{Y_t}$ is the tangential part of $\partial_t$ onto the tangent plane  $TY_t$ of $Y_t$. 

\begin{itemize}
    \item Note that
    \[\nu_S = -\cos(J_t) \nu_{Y_t} + \sin(J_t) \nu_{Z_t}(z),\]
    we have
    \[g(\nabla^{Y_t}{ \varphi},{ \nu_S}) = \sin(J_t) \cdot \frac{\partial \varphi}{\partial \nu_{Z_t}}.\]
    and
    \[g(\nabla_{\partial_t^Y} \nu_{Y_t}, \nu_S) = \sin(J_t(z)) \cdot g(\nabla_{\partial_t^Y} \nu_{Y_t}, \nu_{Z_t}).\]

    \item Note that $\partial_t^{Y_t} = \partial_t^{Z_t} + \varphi \cot(J_t)\cdot \nu_{Z_t}$ where $\partial_t^{Z_t}$ is the tangential part of $\partial_t$ onto $Z_t$, we obtain that
    \begin{align*}
       & g(\nabla_{\partial_t^{Y_t}} \nu_{Y_t}, \nu_S)\\
       =& \sin(J_t) \cdot g(\nabla_{\partial_t^{Y}}\nu_{Y_t}, \nu_{Z_t})\\
       =&
       \sin(J_t) \cdot  g(\nabla_{\partial_t^{Z_t}} \nu_{Y_t}, \nu_{Z_t}) + \cos(J_t) \cdot g(\nabla_{\nu_{Z_t}}\nu_{Y_t}, \nu_{Z_t})\\
       =&
       \sin(J_t)\cdot  {g(\nabla_{\partial_t^{Z_t}} \nu_{Y_t}, \nu_{Z_t})} + \cos(J_t) \cdot A_{\nu_Y}(\nu_Z, \nu_
       Z).
    \end{align*}
    

    \item Note that 
    \[\nu_{Y_t} = \cos(J_t) \cdot \nu_S + \sin(J_t)\cdot  \n_t, \ \ \nu_{Z_t} = -\cos(J_t)\cdot \n_t + \sin(J_t)\cdot \nu_S,\] 
    and
    \[\partial_t = \partial_t^{Z_t} + \frac{\varphi}{\sin(J_t)}\cdot  \n_t, \]
    we obtain
    \begin{itemize}
    \item \begin{align*}
    & g(\nu_{Y_t}, \nabla_{\partial_t} \nu_S)\\
    =& g(\cos(J_t)\cdot \nu_S + \sin(J_t)\cdot  \n_t, \nabla_{\partial_t^{Z_t} + \frac{\varphi}{\sin(J_t)} \cdot \n_t} \nu_S )\\
    =&\sin(J_t)\cdot g(\n_t, \nabla_{\partial_t^{ Z_t}} \nu_S) + g(\n_t, \nabla_{\n_t}\nu_S)\cdot  \varphi \\
    =&\sin(J_t)\cdot g(\n_t, \nabla_{\partial_t^{ Z_t}} \nu_S) + A_{\nu_S}(\n_t, \n_t)\cdot \varphi.
    \end{align*}
  \item       
\begin{align*}
    & g(\nabla_{\partial_t^{Z_t}} \nu_{Y_t}, \nu_{Z_t})\\
    =&g(\nabla_{\partial_t^{Z_t}}\left(\cos(J_t)\nu_S + \sin(J_t) \n_t\right), -\cos(J_t)\n_t + \sin(J_t)\nu_S)\\
    =& -\cos^2(J_t) \cdot g(\nabla_{\partial_t^{Z_t}} \nu_S, \n_t) + \sin^2(J_t)\cdot g(\nabla_{\partial_t^{Z_t}} \n_t, \nu_S) - \nabla_{\partial_t^{Z_t}} J_t(z)\\
    =& -g(\nabla_{\partial_t^{Z_t}} \nu_S, \n_t)  - \nabla_{\partial_t^{Z_t}} J_t(z).
\end{align*}
 \end{itemize}
  \end{itemize} 
 Hence, we reach
 \begin{align*}
        &\frac{d}{dt}\Big|_{t=0} \cos(J(z))\\
        =& \sin(J)\cdot H_{Z_t} - H_S - \cos(J) \cdot H_{Y}+ \sin(J) \cdot \frac{\partial \varphi}{\partial \nu_Z}  +\sin(J) \cdot  \nabla_{\partial_t^{Z_t}} J_t.
 \end{align*}
 
Moreover,
\begin{align*}
    &\frac{d}{dt}\Big|_{t=0}\int_{Z} \frac{\cos(J_t) - \mu_{\partial}}{\sin(J_t)} \cdot \varphi\,\mr d\mathcal{H}^{n-2}_g\\
    =&\int_Z \left(  H_Z - \frac{H_S}{\sin J} -  (\cot J) H_Y\right) \varphi^2 + \left(\frac{\partial \varphi}{\partial{\nu_Z}} + \nabla_{\partial_t^Z}{J_t(z)}- \frac{\nabla_{\partial_t} \mu_{\partial}}{\sin J}\right)  \varphi \,\mr d\mathcal{H}_g^{n-2}\\
    =&\int_Z \left(  H_Z - \frac{H_S}{\sin J} -  (\cot J) H_Y - \frac{1}{\sin^2J}\frac{\partial \mu_\partial}{\partial \n} \right)\varphi^2 + \left(\frac{\partial \varphi}{\partial{\nu_Z}} + 2\nabla_{\partial_t^Z}J\right)  \varphi \,\mr d\mathcal{H}_g^{n-2}.
\end{align*}

Note that
\[-\int_Y \varphi \Delta \varphi = \int_Y |\nabla \varphi|^2 - \int_Z\frac{\partial \varphi}{\partial \nu_Z} \varphi,\] we obtain
      \begin{align*}
       \mathcal{A}^{\prime\prime}(0) =&\int_Y |\nabla \varphi|^2  - \left(\Ric_g(\nu_Y, \nu_Y) + \|A_Y\|^2 + \partial_{\nu_Y} \mu \right)\cdot \varphi^2 d\mathcal{H}_g^{n-1}  \\
       + & \int_Z \left(  H_Z - \frac{H_S}{\sin J} -  (\cot J)H_Y - \frac{1}{\sin^2 J}\frac{\partial \mu_\partial}{\partial \n} \right)\varphi^2   + 2(\nabla_{\partial_t^Z}J)\cdot  \varphi \,\mr d\mathcal{H}_g^{n-2}.
   \end{align*}
   
If $\partial_t^Z = 0$, we obtain,
      \begin{align*} 
       \mathcal{A}^{\prime\prime}(0) =&\int_Y |\nabla \varphi|^2  - \left(\Ric_g(\nu_Y, \nu_Y) + \|A_Y\|^2 + \partial_{\nu_Y} \mu \right) \cdot \varphi^2 \,\mr d\mathcal{H}_g^{n-1}  \\
       + & \int_Z \left(  H_Z - \frac{H_S}{\sin J} -  (\cot J)H_Y - \frac{1}{\sin^2 J}\cdot \frac{\partial \mu_\partial}{\partial \n} \right)\cdot \varphi^2 \,\mr  d\mathcal{H}_g^{n-2}.
   \end{align*}
   \end{proof}
\end{lemma}

{
Note that the last second variation formula in Lemma \ref{lemma: second_variation_formula} above requires $\partial_t^Z=0$. However, any normal vector field can be extend to this kind of admissible vector fields.
\begin{lemma}\label{lem:vector extension}
With notations as above, for given $\varphi\in C^\infty(Y)$, there exists a vector $X$ in $M$ such that
    \begin{itemize}
        \item $X\cdot \nu_Y=\varphi$ for any given $\varphi\in C^\infty(Y)$;
        \item $X|_{\partial M}\in T(\partial M)$;
        \item $X|_{\partial Y}$ is normal to $\partial Y$.
    \end{itemize}
\end{lemma}
\begin{proof}
Recall that $\n$ is the unit outward normal vector field of $\partial Y$ in $\partial M$. Let $\wti \nu_Z$ be the vector field on $Y$ such that $\wti \nu_Z|_{\partial Y}=\nu_Z$. Consider the vector field $X=\varphi\nu_Y+(\varphi \tan J_t) \wti \nu_Z$. Obvisouly, $X|_{\partial Y}$ is parallel to $\n$ on $\partial Y$ and $X\cdot \nu_Y=\varphi$ on $Y$. One can extend it to be a vector field on $M$ satisfying all the conditions.
\end{proof}
}

\section{Maximum principal of  the capillary mu-bubble } \label{sec: MP}

In this section, we will detail the maximum principal (inspired by White \cite{White-Maximum-priciple}) around the artificial corner of the capillary $\mu$-bubble, which forms part of the proof of Lemma \ref{lemma: existence_mu_bubble} in Section \ref{sec: mu-bubble}. 

\begin{claim}\label{claim: mp}
    With the same notations and assumptions as in Lemma \ref{lemma: existence_mu_bubble}. If $\{\Omega_k\}_{k=0}^\infty$ is a minimizing sequence of $\mathcal A_c$, then there exists an open neighborhood $\mathcal{B}_i \subset M$ of $B_i$ ($i=1,2$) such that
    \begin{equation*}
        \mathcal{A}_c(\Omega_k \cup \mathcal{B}_1\setminus \mc B_2) <  \mathcal{A}_c(\Omega_k).
    \end{equation*}
\end{claim}

\begin{proof}


 Without loss of generality, we can assume that each $\Omega$ in the minimizing sequence $\{\Omega_k\}_{k=0}^\infty$ has a smooth boundary. Now, we focus on the case when $i=1$. In other words, adding a neighborhood $\mathcal{B}_1$ of the bottom part $B_1$ to $\Omega$ will result in a decrease in the energy $\mathcal{A}_c$. A similar argument applies to the top part $B_2$.
 

We first isometrically embed $M$  into a closed Riemannian manifold $\wti M$ of the same dimension with $M$. 
Denote by $r_0$ the injective raduis of $\partial M$ in $\wti M$. 
Let $\nu_S$ be the unit, outward vector of $\partial M$. Consider the following family 
\[S_{s,t}(x) := \{\exp_x\big((s-tf(x))  \nu_S(x)\big); x\in \partial M\}, \quad  s,t\in (-r_0/4,r_0/4),  \] 
where $\exp$ is the exponential map in $\wti M$, and $f\in C^\infty(\partial M)$ satisfy the following conditions:
\begin{itemize}
    \item $0\leq f<2$ everywhere, $f|_{\partial B_1}=1$, $f|_{B_1\setminus \partial B_1}>1$; $f|_{M\setminus \oli B_1}<1$;
    \item  $f=0$ outside a small neighborhood of $B_1$;
    \item $\nabla f(x)\neq 0$ for all $x\in \partial B_1$.
\end{itemize}
Then $S_{s,t}$ is an embedded surface in $\wti M$ and bound a domain that intersects $M$ (we continuously choose domains so that $S_{0,0}$ bounds $M$). Denote by $\nu_{s,t}$ the unit outward normal vector field of $S_{s,t}$. By the assumption of $\nabla f\neq 0$ on $B_1$, we get that $S_{s,t}$ intersects $B_1$ transversely around $\partial B_1$ for $(s,t)\neq (0,0)$. Indeed, for $s_1\neq 0$, $t_1\in [0,r_0)$ and $x_1\in \partial B_1$ with $-s_1f(x_1)+t_1\zeta(x_1)=0$, then $s_1=t_1$. A standard computation gives that 
\[   \nu_{s_1,s_1}(x_1)\cdot\nu_{0,0}(x_1)= \frac{1}{\sqrt{1+|t_1\nabla f|^2}}<1=\mu_{\partial}\quad \text{ for all } x\in \partial B_1.\]
Now we pick $0<s'<r_0/8$ such that for all $t\in (0,2s_0)$, 
\[H|_{S_{s',t}}=\dv_{S_{s',t}}\nu_{s',t}\geq \frac{1}{2}\inf_{x\in\partial M}H_{\partial M}.\]
Observe that $S_{s',s'}\cap \partial M=\partial B_1$. Then pick $\delta>0$ small enough such that for all $t\in[s',t']$ ($t':=s'+\delta$), 
\begin{gather*}
     \nu_{s',t}(y)\cdot \nu_{0,0}(y)<\mu_{\partial} \quad \text{ whenever } y\in S_{s',t}\cap \partial M.   
\end{gather*}  
Let $\mc B_1=\cup_{t\in[0,t']}S_{s',t}\cap M$. Since $f\geq 0$ everywhere, then the vector field defined by
\[   \nu(x):=\nu_{s',t} \quad \text{ whenever } x\in S_{s',t}\]
is smooth. 
Note that $s'-0\cdot f >0$, $s'-s'f < 0$ on $B_1\setminus \partial B_1$. We conclude that $\mc B_1$ contains a small neighborhood of $B_1$. 

For any $\Omega\supset B_1$ that intersects $S_{s',s'}$ transversely, denote by $V:=\oli{\mc B_1\setminus \Omega}$. Then by the divergence theorem
\begin{align*}
      & \ \ \ \ \mathcal{A}_c(\Omega \cup \mathcal{B}_1) - \mathcal{A}_c(\Omega) \\
      &=\mathcal{H}_g^{n-1}(S_{s',s'}\cap \partial V)  - \mc H^{n-1}_g(\partial V\setminus (S_{s',s'}\cup\partial M))-\int_{\partial M \cap \partial V} \mu_{\partial}\\
      &\leq \int_{ S_{s',s'}\cap \partial V} \nu\cdot \nu_{s',s'} - \int_{\partial V\setminus (S_{s',s'}\cup\partial M) } \nu \cdot (-v_{\partial \Omega})  - \int_{\partial M \cap \partial V} \mu_{\partial} \\
      & \leq  \int_{V} \Div\nu + \int_{\partial M \cap \partial V} \nu \cdot \nu_{\partial M} - \mu_{\partial}\\
       &\leq  \int_{V} \Div \nu < 0.
\end{align*}

\end{proof}

\section{Shi-Tam inequality and its extension} \label{sec: Shi-Tam}

The proof of the rigidity Theorem \ref{mainthm: scalar_mean_curvature_comparsion}, \ref{thm: generalized_scalar_mean_comparsion}, and \ref{thm:Lipschitz} utilize the Shi-Tam inequality and its extension. Therefore, for the readers' convenience, we will review the Shi-Tam inequality for the case $n=3$ in \cite{Shi_Tam}, as well as its extension for $4 \leq n \leq 7$ from \cite{Shi_Tam_extension}.

\begin{theorem}[Shi-Tam, see \cite{Shi_Tam}*{Theorem 1}] \label{thm: shi-tam_inequality}
    Suppose that $(M^n, \partial M, g)$ is a smooth,  compact, \textbf{spin} Riemannian manifold with nonnegative scalar curvature $\Sc_g \geq 0$ and mean convex boundary $H_{\partial M} > 0$. If $\partial M$ consists of connected components $\{\Sigma_i\}_i^k$ and each connected component $\Sigma_i$ can be isometrically embedded in $\mathbb{R}^n$ as a strictly convex hypersurface, and we denote the mean curvature by $\hat{H}_{\Sigma_i}$ in $\mb R^n$, then
    \begin{equation*}
        \int_{\Sigma_i} H_{\Sigma_i} \,\mr d\mathcal{H}_g^{n-1} \leq \int_{\Sigma_i} \hat{H}_{\Sigma_i}\,\mr d\mathcal{H}_g^{n-1}.
    \end{equation*}
    Moreover, the equality holds for some boundary $\Sigma_i$ if and only if $(M, \partial M, g)$ is isometric to a domain in $\mathbb{R}^n$. 
\end{theorem}

As the dimension $n=3$, all three-dimensional manifolds are spin, and the requirement of the embeddings for the boundaries in Theorem \ref{thm: shi-tam_inequality} is equivalent to positive Gauss curvature on the boundary $\partial M$. Hence, for any smooth, compact three-dimensional Riemannian manifold with nonnegative scalar curvature. If the boundary $\partial M$ has positive Gauss curvature, then Shi-Tam inequality (Theorem \ref{thm: shi-tam_inequality}) holds. 

However, in dimension $n > 3$, there are no analogous intrinsic conditions on the boundary of $(M^n, \partial M, g)$ that guarantee that its components embed
isometrically into $\mb  R^n$. Eichmair-Miao-Wang extends Shi-Tam inequality as follows.

\begin{theorem}[Eichmair-Miao-Wang, see \cite{Shi_Tam_extension}*{Theorem 2}]
    The conclusion of Theorem \ref{thm: shi-tam_inequality} remains valid if the assumption
that every boundary component embeds as a strictly convex hypersurface in
$\mb R^n$ is relaxed to the requirement that the boundary of $(M, \partial M, g)$ has positive
scalar curvature and that each boundary component is isometric to a mean convex, star-shaped hypersurface in $\mb R^n$. Moreover, the spin assumption can
be dropped in dimensions $3 \leq n \leq 7$.
\end{theorem}

Finally, we would like to recall the total mean curvature conjecture as follows.

\begin{conjecture}[Gromov, see \cite{Gromov_four_lectures}*{section 3.12.2}]\label{conj: total_mean_curvature}
    Suppose that $(M^n, \partial M, g)$ is a smooth, compact Riemannian manifold with scalar curvature $\Sc_g \geq -\sigma$ on $M$. If $H_{\partial M}$ is the mean curvature of $\partial M$ in $M$, then there exists a constant $c(\sigma, g_{\partial M})$ such that
    \begin{equation*}
        \int_{\partial M }H_{\partial M} \,\mr  d \mathcal{H}_g^{n-1} \leq c(\sigma, g_{\partial M}).
    \end{equation*}
\end{conjecture}
Hence, Theorem \ref{mainthm: scalar_mean_curvature_comparsion} can be viewed as a weaker variant of Conjecture \ref{conj: total_mean_curvature}. 

\bibliographystyle{amsalpha}
\bibliography{scalar_mean.bib}

\end{document}